\documentclass{article}
\usepackage{amsmath,amssymb,amsthm,a4wide, hyperref}

\newtheorem{lemma}{Lemma}[section]
\newtheorem{theorem}[lemma]{Theorem}
\newtheorem{remark}[lemma]{Remark}
\newtheorem{corollary}[lemma]{Corollary}
\newtheorem{definition}[lemma]{Definition}
\newtheorem{proposition}[lemma]{Proposition}

\begin{document}
\title{Counting torsion points on subvarieties of the algebraic torus}
\author{Gerold Schefer}
\maketitle

\paragraph{Abstract:}We estimate the growth rate of the function which counts the number of torsion points of order at most $T$ on an algebraic subvariety of the algebraic torus $\mathbb G_m^n$ over some algebraically closed field. We prove a general upper bound which is sharp, and characterize the subvarieties for which the growth rate is maximal. For all other subvarieties there is a better bound which is power saving compared to the general one. Our result includes asymptotic formulas in characteristic zero where we use Laurent's Theorem, the Manin-Mumford Conjecture. However, we also obtain new upper bounds for $K$ the algebraic closure of a finite field.

\section{Introduction}
\paragraph{Main results.} Let $n\ge 1$ be an integer and $K$ be an algebraically closed field. We denote by $\mathbb G_m^d$ the $n$-dimensional algebraic torus with base field $K$. We will identify $\mathbb G_m^n(K)$ with $(K\setminus \{0\})^n$, the group of its $K$-points. We denote by $(\mathbb G_m^n)_{\mathrm{tors}}$ the subgroup of points of finite order therein, which we call torsion points. A linear torus $G<\mathbb G_m^n$ is an irreducible algebraic subgroup and if $\boldsymbol\zeta \in(\mathbb G_m^n)_{\mathrm{tors}}$ is a torsion point and $G$ is a linear torus, we call $\boldsymbol\zeta G$ a torsion coset. We often identify a subvariety of $\mathbb G_m^n$ with the set of its $K$-points.

Lukas Fink considered the following question in his master thesis \cite{Fink}. Let $X=\{(x,y)\in \mathbb G_m^2(\overline{\mathbb F}_2) : x+y=1\}$ and $c(T)=\#\{\mathbf x\in X: \mathrm{ord}(\mathbf x)\le T\}$. What can we say about the asymptotic growth of $c(T)$ as $T\to \infty$? Considering the first coordinate yields $c(T)\le \#\{x\in \mathbb G_m:\mathrm{ord}(x)\le T\} \ll T^2$. He shows in Proposition 2.7 that $c(T)=o(T^2)$. Fink's result is therefore better than the trivial bound. Based on computations he conjectures that $c(T)\ll T$. Our first result, Theorem \ref{thm:fink} below, includes a power saving over the trivial bound $c(T)\ll T^2$.

We consider the average rather than the number of torsion points of order exactly $n$ since the latter behaves erratic. In Fink's example the set of $n\in\mathbb N$ such that there is no torsion point of order $n$ has natural density one. The number of points on Fink's curve whose order divides $n$ is also bounded by a constant times $n^{1-1/560}$ if (4) in \cite{Ali} holds for the set of points $x\in\mathbb F_{2^{\varphi(n)}}$ such that $x^n=1=(x+1)^n$. On the other hand, this number is $n-1$ for all $n=2^k-1$.

An algebraic subset $X\subset \mathbb G_m^n(K)$ is the zero set of finitely many Laurent polynomials $P_1,\dots, P_r\in K[X_1^{\pm 1},\dots, X_n^{\pm 1}]$ in $\mathbb G_m^n(K)$. The natural generalisation of the question of Fink is as follows. Take an algebraic set $X\subset \mathbb G_m^n(K)$, and denote by $X_T=\{\boldsymbol\zeta\in X\cap (\mathbb G_m^n)_{\mathrm{tors}}: \mathrm{ord}(\boldsymbol\zeta)\le T\}$ the set of torsion points in $X$ of order bounded by $T$. How fast does $\#X_T$ grow? 

In characteristic zero, the Manin-Mumford Conjecture for $\mathbb G_m^n$, a theorem of Laurent \cite{ML} reduces this problem to the case where $X$ is a torsion coset as we will see below in Corollary \ref{cor:char0}. Altough many of our results are true in any characteristic, we highlight the case where $K$ is an algebraic closure of a finite field. In this case all points of $\mathbb G_m^n$ have finite order. To illustrate the kind of results we get and compare them to the work of Fink, we look at the curve defined by $x+y=1$:

\begin{theorem}\label{thm:fink}Let $p$ be a prime and $X=\{(x,y)\in \mathbb G_m^2(\overline{\mathbb F}_p) : x+y=1\}$. Then we have $\#X_T\le 16 T^{3/2}$ for all $T\ge 1.$\end{theorem}

To state the main result, we have to introduce more terminology. An algebraic subset $X\subset \mathbb G_m^n$ is called admissible if it is finite or does not contain a top dimensional torsion coset. The stabilizer $\mathrm{Stab}(X)$ of an algebraic subset $X\subset \mathbb G_m^n$ is the set of $\mathbf g\in\mathbb G_m^n$ such that $\mathbf gX\subset X$. It is well known that the stabilizer is an algebraic subgroup of $\mathbb G_m^n$. The main result is
\begin{theorem}\label{thm:main}Let $K$ be an algebraically closed field and $X\subset \mathbb G_m^n$ be irreducible, admissible and of dimension $d$. Let $\delta$ be the dimension of the stabilizer $\mathrm{Stab}(X)$. Then we have $X_T\ll_X T^{d+1-\frac{1}{d-\delta+1}}$ for all $T\ge 1$.\end{theorem} 
As in the example there is an analogue of the trivial bound given by 
\begin{proposition}\label{thm:gen}Let $K$ be an algebraically closed field and $X\subset \mathbb G_m^n$ an algebraic subset of dimension $d$. Then $\#X_T\ll_X T^{d+1}$ for all $T\ge 1$.\end{proposition}
So again we have a power saving with respect to the general bound.

If $C$ is a torsion coset, we can give an asymptotic formula whose main term depends only on the characteristic of the field, the dimension and the order $\mathrm{ord}(C)=\min\{\mathrm{ord}(\boldsymbol\eta): \boldsymbol\eta\in C\cap (\mathbb G_m^n)_{\mathrm{tors}}\}$. In the theorem below $\zeta$ denotes the Riemann zeta function.
\begin{theorem}\label{thm:countCoset}Let $K$ be an algebraically closed field of characteristic $p$ and let $C\subset \mathbb{G}_m^n$ be a torsion coset of dimension $d$ and order $g$. Then if $p=0$ we have $$\# C_T =\frac{T^{d+1}}{(d+1)\zeta(d+1)g}+\begin{cases}O_g(T\log(T))&\text{if }d=1\\O_{d,g}(T^d) &\text{if } d\ge 2\end{cases}\text{ as }T\to\infty$$
and if $p>0$ we have
$$\#C_T =\frac{(p-1)T^{d+1}}{(p-p^{-d})(d+1)\zeta(d+1)g}+\begin{cases}O_g(T\log(T))&\text{if }d=1\\O_{d,g}(T^d) &\text{if } d\ge 2\end{cases}\text{ as }T\to\infty.$$\end{theorem}
Theorem \ref{thm:countCoset} implies that the exponent $d+1$ in Proposition \ref{thm:gen} is optimal. We also see that torsion cosets are those subvarieties which in some sense contain the most torsion points if we compare among varieties of the same dimension.

Recall that in characteristic zero the Manin-Mumford Conjecture for $\mathbb G_m^n$ reduces the general case to torsion cosets and again we are able to give an asymptotic formula. Let $X\subset \mathbb G_m^n$ be an algebraic set. There exist torsion cosets $C_1, \dots, C_m$ such that $\overline{X\cap(\mathbb G_m^n)_{\mathrm{tors}}}=\bigcup_{i=1}^m C_i$ is the decomposition into irreducible components. Set $$a(X)=\max_{1\le i\le m} \dim(C_i)\text{ and }b(X)=\frac{1}{(a(X)+1)\zeta(a(X)+1)}\sum_{\substack{i=1\\ \dim(C_i)=a(X)}}^m \frac{1}{\mathrm{ord}(C_i)}.$$ Then we have
\begin{corollary}\label{cor:char0}Let $K$ be an algebraic field of characteristic zero, let $X\subset \mathbb{G}_m^n$ be algebraic, $a=a(X)$ and $b=b(X)$. Then we have $\#X_T \ll_X 1$ if $a=0$ and $$\#X_T = b T^{a+1}+\begin{cases}O_X(T\log(T)) & \text{if }a=1\\
O_X(T^a)	 & \text{if }a\ge 2\end{cases}\text{ as }T\to\infty.$$\end{corollary}

What about lower bounds for $\# X_T$? Fink proved in Proposition 2.7 in \cite{Fink} that $c(T)\ge T/2-2$ for all $T\in\mathbb N$. Using the famous theorem of Lang-Weil, we can prove 
\begin{theorem}\label{thm:charp}Let $p$ be a prime and $n\in\mathbb N$. Let $X\subset \mathbb G_m^n(\overline{\mathbb F}_p)$ be algebraic of dimension $d$. Then there exist $c=c(X)>0$ and $T_0\ge 1$ such that $cT^d \le \#X_T$ for all $T\ge T_0$.\end{theorem}
Hence if $X$ is admissible, the correct exponent, if it exists, is in $[d,d+1-\frac{1}{d-\delta+1}]$, where $\delta=\dim(\mathrm{Stab}(X))$. In the example of Fink, numerical analysis suggests that the smallest exponent $d=1$ is possible.

\paragraph{Overview of the proofs.}Let us give two examples. First we look at the curve $C=\{(x,y)\in\mathbb G_m^2: x+y=1\}$ and prove Theorem \ref{thm:fink}. Let $T\ge 1$ and assume that $\boldsymbol\zeta=(\zeta,\xi)\in C_T$ is of order $N$. By Minkowski's first theorem there exists $\mathbf a=(a,b)\in\mathbb Z^2\setminus \{0\}$ such that $\boldsymbol\zeta^{\mathbf a}:=\zeta^a\xi^b=1$ and $|\mathbf a|:=\max\{|a|,|b|\}\le N^{1/2}$. Since $\xi=1-\zeta$ we have $\zeta^a(1-\zeta)^b= 1$. After multiplying by suitable powers of $\zeta$ and $1-\zeta$ we see that $\zeta$ is a root of a nonzero polynomial of degree at most $|a|+|b|$. Hence the number of $\boldsymbol\zeta\in C$ such that $\boldsymbol\zeta^{\mathbf a}=1$ is at most $2|\mathbf a|$. Thus we have the bound $\#C_T\le 2\sum_{|\mathbf a|\le T^{1/2}}|\mathbf a|$. For a fixed $k\ge 1$ there are $4(2k+1)-4=8k$ vectors $\mathbf a\in\mathbb Z^2$ such that $|\mathbf a|=k$. Thus we find 
$$\#X_T\le 16 \sum_{k=1}^{T^{1/2}}k^2 \le 16T^{3/2}.$$

Now let us consider $Y=\{(x,y,z)\in\mathbb G_m^3: x+y=1\}$. Note that Stab$(Y)=\{\mathbf 1\}\times \mathbb G_m$ is one dimensional. Hence the exponent in Theorem \ref{thm:main} equals $5/2$.
Then we observe that any $(\zeta_1,\zeta_2,\zeta_3)\in Y\cap(\mathbb G_m^3)_\mathrm{tors}$ is contained in the torsion coset $\{\zeta_1\}\times \{\zeta_2\}\times \mathbb G_m\subset Y.$ The order of such a torsion coset is given by $\mathrm{ord}(\zeta_1,\zeta_2)$. Thus any $\boldsymbol\zeta\in Y_T$ is contained in a torsion coset $\{\boldsymbol\xi\}\times \mathbb G_m$ for some $\boldsymbol\xi\in C_T$. Therefore we can bound $\#Y_T\le \sum_{\boldsymbol\xi\in C_T}\#(\{\boldsymbol\xi\}\times \mathbb G_m)_T$. Denote by $g$ the order of $\boldsymbol\xi\in C_T$. We claim that $\#(\{\boldsymbol\xi\}\times \mathbb G_m)_T\le T^2/g$. If $\zeta\in \mathbb G_m$ is of order $n$, the order of $(\boldsymbol\xi, \zeta)$ is the least common multiple of $g$ and $n$. Thus $(\boldsymbol\xi, \zeta)\in (\{\boldsymbol\xi\}\times \mathbb G_m)_T$ if and only if $n\le \gcd(g,n)T/g$. The number of elements of order $n$ in $\mathbb G_m$ is given by Euler's totient $\varphi(n)$ except if the characteristic $p>0$ divides $n$. Since in positive characteristic the order is always coprime to $p$, there are no elements of order $n$, whenever $p$ divides $n$. Thus in any case there are at most $\varphi(n)$ elements of order $n$ in $ \mathbb G_m$. We therefore have $$\#(\{\boldsymbol\xi\}\times \mathbb G_m)_T\le \sum_{1\le n\le \gcd(g,n)T/g}\varphi(n)=\sum_{d|g}\sum_{n: \gcd(g,n)=d, n\le dT/g}\varphi(n).$$ 
If $d|g$ and $n$ is as in the inner sum, we can write $n=kd$ for some $1\le k\le T/g$. And since $\varphi(n)=n\prod_{p\mid n}(1-p^{-1})$ we see that $\varphi(ab)\le a\varphi(b)$ for all $a,b\in\mathbb N$. It is also well known, that $\sum_{d|n}\varphi(d)=n$ for all $n\in\mathbb N$. Therefore we get $$\#(\{\boldsymbol\xi\}\times \mathbb G_m)_T\le \sum_{d|g}\sum_{1\le k\le T/g} \varphi(kd)\le \sum_{d|g}\varphi(d)\sum_{1\le k\le T/g} k \le g(T/g)^2=T^2/g.$$ Let $a_n$ denote the number of elements of order $n$ in $C$. Then we can bound $\#Y_T\le T^2\sum_{n=1}^T a_n /n$ and we know that $\sum_{n=1}^N a_n\le 16N^{3/2}$ for every $N\ge 1$. To get sums of this kind, we write $1/n=1/T+\sum_{k=n}^{T-1}(\frac{1}{k}-\frac{1}{k+1})=1/T+\sum_{k=n}^{T-1}\frac{1}{k(k+1)}$ for all $n<T$. Thus we find $$\sum_{n=1}^T a_n /n = \sum_{k=1}^{T-1}\frac{1}{k(k+1)}\sum_{n=1}^k a_n+1/T\sum_{n=1}^T a_n\le 16\sum_{k=1}^{T-1}k^{3/2-2}+16T^{3/2-1}\le 48T^{1/2}$$ using that $\int x^{-1/2}dx = 2x^{1/2}$ and hence $Y_T\le 48T^{5/2}$.\\

Now let us consider a hypersurface $X\subset \mathbb G_m^n$. We divide $X$ into two parts: the union of all torsion cosets contained in $X$ of positive dimension and its complement $X^*$. To bound $\#X^*_T$ we generalise the proof for $C$: Instead of Minkowski's first theorem we can use the second theorem to find short linearly independent vectors $\mathbf a_1,\dots, \mathbf a_{n-1}\in\mathbb Z^n$ such that $\boldsymbol\zeta^{\mathbf a_i}=1$. The $\mathbf x\in\mathbb G_m^n$ such that $\mathbf x^{\mathbf a_i}=1$ for all $i=1,\dots, n-1$ are given by torsion cosets $\{\boldsymbol\xi t^{\mathbf u^\top}: t\in\mathbb G_m\}$ for some $\mathbf u\in\mathbb Z^n\setminus \{0\}$ and finitely many $\boldsymbol\xi\in(\mathbb G_m^n)_{\mathrm{tors}}$; we refer to Section \ref{sec:not} for our notation. Let $P\in K[X_1^{\pm1},\dots, X_n^{\pm 1}]$ be such that $X$ is the vanishing locus of $P$ in $\mathbb G_m^n$. Then $\boldsymbol\xi t^{\mathbf u^\top}$ is in $X$ if and only if $Q(t)=P(\boldsymbol\xi t^{\mathbf u^\top})=0$. If $Q$ is not the zero polynomial, the number of such $t$ is bounded in terms of the degree of $P$ and $|\mathbf u|$. This is always the case if $\boldsymbol\xi t^{\mathbf u^\top}\in X^*$. Here we use the fact that $X$ is a hypersurface. Summing over all possible $(n-1)$-tuples $\mathbf a_1,\dots, \mathbf a_{n-1}$ gives us a bound $X^*_T\ll_X T^{n-1/n}$.

For the union of positive dimensional torsion cosets we show that if we consider the maximal cosets, then the number of corresponding algebraic groups is finite and every torsion coset in $X$ is contained in one of them. However in positive characteristic it is possible that for a fixed linear torus $G$ there are infinitely many torsion cosets $\boldsymbol\zeta G$ which are contained in $X$. So let us fix $G$. To control the number of cosets $\boldsymbol\zeta G$ which are contained in $X$, we map $\boldsymbol\zeta G$ to $\boldsymbol\zeta^L$ for some matrix $L$ such that $\mathbf g^L=\mathbf 1$ for all $\mathbf g\in G$. Note that $\mathrm{ord}(\boldsymbol\zeta^L)\le \mathrm{ord}(\boldsymbol\zeta G)$. We can find equations for the image and use induction on $n$. Thus we have a bound for the number of cosets of $G$ contained in $X$. Together with a bound for cosets we get the claimed result for hypersurfaces.

The generalisation to arbitrary algebraic sets follows by projecting $X$ to some coordinates and doing algebraic geometry. 

For Theorem \ref{thm:countCoset} we can assume that $G=\{\boldsymbol\zeta\}\times \mathbb G_m^d$ for some $\boldsymbol\zeta\in(\mathbb G_m^{n-d})_{\mathrm{tors}}$ of order $g$, say. Since the order of $(\boldsymbol{\zeta,\xi})$ is $\mathrm{lcm}(g,\mathrm{ord}(\boldsymbol\xi))=\frac{g}{(g,\mathrm{ord}(\boldsymbol\xi))}\mathrm{ord}(\boldsymbol\xi)$, we have to count the $\boldsymbol\xi\in(\mathbb G_m^d)_{\mathrm{tors}}$ whose order is bounded and belongs to a fixed congruence class modulo $g$. With the M\"obius inversion formula one can see that the number of elements of order $n$ is given by Jordan's totient $J_d(n)=\mu\star I_d(n)$, where $\mu$ is the M\"obius function and $I_d(n)=n^d$. Once we have the asymptotics of $J_d$ over arithmetic progressions, Theorem \ref{thm:countCoset} is not hard to deduce. In positive characteristic the orders are always coprime to $p$. If $n$ is coprime to $p$, we still have $J_d(n)$ elements of order $n$. Thus the case $p>0$ follows from the case $p=0$.

The proofs of Corollary \ref{cor:char0} and Theorem \ref{thm:charp} are short and straight forward.

\paragraph{Organisation of the article}

The algebraic subgroups of $\mathbb G_m^n$ play an important role in the results as well as in the proofs. Therefore we need a characterisation of them. This is already known in characteristic zero (Theorem 3.2.19 in \cite{BG}). In positive characteristic it is possibly known. But we could not pin down a suitable reference and have therefore added an appendix.

In section \ref{sec:jord} we prove Theorem \ref{thm:countCoset}.

The next goal is to estimate $\#X^*_T$. The most technical part is to bound the degree of $Q(t)$. Some known preparations are done in the second appendix and used to prove the bound in section \ref{sec:out}.

For the proof of Theorem \ref{thm:main} when $X$ is a hypersurface, we first deal with common roots and admissibility in section \ref{sec:admiss}, then we establish an explicit bound when $X$ is a coset. In section \ref{sec:hyp} we give the proofs of Theorems \ref{thm:main} and \ref{thm:gen} when $X$ is a hypersurface.

Finally we are able to prove Theorems \ref{thm:main} and \ref{thm:gen} in full generality in the next section.

In the last two sections we establish Corollary \ref{cor:char0} and Theorem \ref{thm:charp}.

\paragraph{Acknowledgements} I want to thank my advisor Philipp Habegger for good discussions and advice, and Pierre le Boudec for his help concerning the sum of Jordan's totient in arithmetic progressions. This will be part of my PhD thesis. I have received funding from the Swiss National Science Foundation grant number 200020\_184623.

\section{Notation}\label{sec:not}
We write $\mathbb{N}=\{1,2,3,\dots\}$ and $K$ denotes always an algebraically closed field of any characteristic. Let $n\in\mathbb{N}$ and $R=K[X_1^{\pm 1},\dots X_n^{\pm 1}]$ the ring of Laurent polynomials in $n$ variables. Let $R^*$ be the units of $R$. Then for $r,s\in\mathbb{N},\mathbf x=(x_1,\dots,x_n)\in (R^*)^n$ and $A=(a_{i,j})\in\mathrm{Mat}_{n\times m}(\mathbb{Z})$ we define $\mathbf x^A=(x_1^{a_{1,1}}\cdots x_n^{a_{n,1}},\dots , x_1^{a_{1,m}}\cdots x_n^{a_{n,m}})\in (R^*)^m$. This definition is compatible with matrix multiplication in the sense that  $\mathbf x^{AB}=(\mathbf x^A)^B$ for all $k,m,n\in\mathbb{N},\mathbf x\in (R^*)^k, A\in\mathrm{Mat}_{k\times m}(\mathbb{Z})$ and $B\in\mathrm{Mat}_{m\times n}(\mathbb{Z})$. The letter $\mathbf X$ denotes the element $(X_1,\dots, X_n)\in (R^*)^n$ and hence we can write a general Laurent polynomial $P\in R$ as $P=\sum_{\mathbf i\in \mathbb{Z}^n} p_{\mathbf i}\mathbf{X^i}$. The support of such a polynomial is the finite set of $\mathbf i\in\mathbb{Z}^n$ such that $p_{\mathbf i}\neq 0$ and we denote it by $\mathrm{Supp}(P)$.\\
For $n\in\mathbb{N}$ we set $\mathbb{G}_m^n(K)=(K\setminus \{0\})^n$ which together with the coordinate-wise multiplication is a group.  We write $\mathbb{G}_m^n(K)_{\mathrm{tors}}$ for the elements of finite order in $\mathbb{G}_m^n(K)$ and call them torsion points. If it is clear what $K$ is, we usually omit $K$ in the notation. Elements $\zeta\in(\mathbb{G}_m)_{\mathrm{tors}}$ are called roots of unity. For $N\in\mathbb{N}$ we write $\mu_N(K)$ or $\mu_N$ for the roots of unity $\zeta$ in $K$ such that $\zeta^N=1$.\\
We denote by $E_n$ the identity matrix in the ring of $n\times n$-matrices $\mathrm{Mat}_{n}(\mathbb{Z})$. We write $\mathbf{0}$ and $\mathbf{1}$ for matrices whose entries are all zero or one respectively. We write $\mathbf e_i$ for the $i$-th vector of the standard basis of $\mathbb{Z}^n$. The $n$ will be clear from the context.
We denote by $\mathbb{A}^n(K)$ the affine space $K^n$ endowed with the Zariski topology.\\
For $\mathbf a=(a_1,\dots, a_n)\in\mathbb{Z}^n$ we denote by $|\mathbf a|=\max_{1\le i\le n}|a_i|$ the maximum of the absolute vaues of its coordinates, where $|\cdot|$ on the right hand side is the usual absolute value. We use the Vinogradov notation $\ll$ and $\ll_{I}$ for some set $I$ of variables.

\section{Jordan's totient function and torsion cosets}\label{sec:jord}
In characteristic zero the $d$-th Jordan's totient function $J_d$ counts the number of points of order $n$ in $\mathbb G_m^d$. Thus it is a natural generalization of Euler's totient function to higher dimensions. In this section we prove an asymptotic formula for the sum of Jordan's totient function in arithmetic progressions and deduce an asymptotic formula for $\#C_T$, where $C$ is a torsion coset.
\begin{definition}Let $\mathcal{A}=\{f:\mathbb{N}\to\mathbb{C}\}$ be the set of arithmetic functions. We say that $f\in \mathcal A$ is completely multiplicative if $f(nm)=f(n)f(m)$ for all $n,m\in\mathbb{N}$. The Dirichlet convolution of $f,g\in\mathcal A$ is defined as $(f\star g)(n)=\sum_{d|n}f(d)g(n/d).$ We can associate to $f\in \mathcal A$ the formal Dirichlet series $\sum_{n=1}^\infty f(n)n^{-s}$, where $s$ is an indeterminate. \end{definition}
\begin{definition}Define $\eta\in\mathcal A$ by $\eta(1)=1$ and $\eta(n)=0$ for all $n\ge 2$ and $\epsilon\in\mathcal A$ by $\epsilon(n) = 1$ for all $n\in\mathbb N$. The M\"obius function $\mu\in \mathcal A$ is nonzero only on squarefree numbers and if $n$ is a product of $r$ distinct primes we have $\mu(n)=(-1)^r$.\\
For $d\in\mathbb N$ we define $I_d\in\mathcal A$ as $n\mapsto n^d$ and Jordan's totient function $J_d\in\mathcal A$ given by $J_d(n)=\mu\star I_d$, a generalisation of Euler's totient function $\varphi=J_1$.\\
Any character $\chi: \left(\mathbb{Z}/m\mathbb Z\right)^*\to\mathbb C^*$ gives rise to $f_\chi\in\mathcal A$ defined by $f_\chi(n)=\chi(n+m\mathbb Z)$ if $(n,m)=1$ and zero otherwise. We will sometimes abuse notation and write $\chi$ for $f_\chi$.\\
The Riemann zeta function $\zeta(s)=\sum_{n=1}^\infty n^{-s}$ is the series corresponding to $\epsilon$ and the Dirichlet $L$-functions are defined by $L(s,\chi)=\sum_{n=1}^\infty \chi(n)n^{-s}$.\end{definition}
\begin{remark}Together with the pointwise addition the Dirichlet convolution endows $\mathcal A$ with the structure of a commutative unitary ring with multiplicative identity $\eta$. If $f\in\mathcal A$ is completely multiplicative we have $f(g\star h)=(fg)\star (fh)$ for all $g,h\in \mathcal A$. Here $fg$ denotes the pointwise multiplication. For all $f,g\in \mathcal{A}$ we have $$\sum_{n=1}^\infty f(n)n^{-s}\sum_{m=1}^\infty g(m)m^{-s}=\sum_{n=1}^\infty (f\star g)(n)n^{-s}$$ as formal series. In particular if all three series converge absolutely for some $s\in\mathbb{C}$ the equality above holds in $\mathbb{C}$.\end{remark}
\begin{remark}M\"obius inversion is equivalent to $\mu\star \epsilon=\eta$. The functions $I_d$ and $f_\chi$ are completely multiplicative for all $d\in\mathbb N$ and characters $\chi$ respectively. The explicit formula for Jordan's totient function reads $J_d(n)=n^d\prod_{p|n}(1-p^{-d})$ where the product is taken over all primes dividing $n$ and as usual the empty product equals one. The Riemann $\zeta$-function as well as the Dirichlet $L$-functions converge absolutely if the real part of $s$ is strictly larger than $1$.\end{remark}

\begin{lemma}\label{lem:sumnk}Let $d,m\in\mathbb{N}$ and $a\in\mathbb{Z}$. Then $$\sum_{\substack{n\le x\\ n\equiv a (m)}} n^d=\frac{x^{d+1}}{(d+1)m} +O_{d,m}\left(x^d\right) \text{ as } x\to\infty.$$\end{lemma}
\begin{proof}Let $b\in\{1,\dots, m\}$ such that $b\equiv a (m)$. Then every $n\le x$ which is congruent to $a$ modulo $m$ is of the form $n=b+lm$ for some $0\le l\le(x-b)/m$. Thus we have $$\sum_{\substack{n\le x\\ n\equiv a (m)}} n^d = \sum_{0\le l\le (x-b)/m}(b+lm)^d=  \sum_{0\le l\le (x-b)/m}\sum_{i=0}^d \binom{d}{i} b^{d-i} (lm)^{i}= \sum_{i=0}^d \binom{d}{i} b^{d-i} m^{i}\sum_{0\le l\le (x-b)/m} l^{i}.$$
By exercise 6.9 in \cite{StRef} we know that $$\sum_{0\le l\le (x-b)/m} l^{i}= \frac{\left(\frac{x-b}{m}\right)^{i+1}}{i+1}+O_d\left(\left(\frac{x-b}{m}\right)^i\right)= \frac{(x/m)^{i+1}}{i+1}+O_d\left(x^i\right).$$ So if $i\le d-1$ then $\sum_{0\le l\le (x-b)/m} l^{i}= O_{d}(x^d)$ and hence $$\sum_{\substack{n\le x\\ n\equiv a (m)}} n^d = m^d \frac{x^{d+1}}{(d+1)m^{d+1}}+O_{d,m}\left(x^d\right)=\frac{x^{d+1}}{(d+1)m} +O_{d,m}\left(x^d\right).\qedhere$$\end{proof}

\begin{lemma}\label{lem:zeta}Let $d,m\in\mathbb{N}$. Then $$\sum_{\substack{ n\le x\\ (n,m)=1}}\frac{\mu(n)}{n^{d+1}}=\frac{m^{d+1}}{\zeta(d+1)J_{d+1}(m)}+O_d\left(x^{-d}\right)\text{ as }x\to\infty.$$\end{lemma}
\begin{proof}Let $\chi_0$ be the principal character modulo $m$, and recall that $\mu\star \epsilon = \eta$. Since $\chi_0$ is completely multiplicative we have $(\mu\chi_0)\star \chi_0=(\mu\chi_0)\star (\epsilon\chi_0)=\chi_0(\mu\star \epsilon)=\chi_0\eta=\eta$. Therefore the infinite sum $S=\sum_{(n,m)=1}\mu(n)n^{-(d+1)}$ satisfies $SL(d+1,\chi_0)=1$. We have $L(d+1,\chi_0)=\zeta(d+1)\prod_{p|m}(1-p^{-(d+1)})=\zeta(d+1)J_{d+1}(m)m^{-(d+1)}$ and hence $S=\frac{m^{d+1}}{\zeta(d+1)J_{d+1}(m)}$.
For the error term we find \begin{align*}\left|\sum_{\substack{n>x\\ (n,m)=1}} \frac{\mu(n)}{n^{d+1}}\right|&\le \sum_{n>x}n^{-(d+1)}\le\sum_{l=\lceil x\rceil -1}^\infty \int_{l}^{l+1} t^{-(d+1)}dt \\
&=\int_{\lceil x\rceil -1}^\infty  t^{-(d+1)}dt = 1/d(\lceil x\rceil -1)^{-d}=O_d(x^{-d}).\qedhere\end{align*}\end{proof}

\begin{theorem}\label{thm:Jord}Let $d,m\in\mathbb{N}$, $a\in\mathbb{Z}$ and $g=(a,m)$. Then $$\sum_{\substack{n\le x\\ n\equiv a (m)}} J_d(n)=\frac{m^dJ_d(g)x^{d+1}}{(d+1)\zeta(d+1)g^dJ_{d+1}(m)}+\begin{cases}O_{m}(x\log(x)) &\text{if }d =1\\ O_{d,m}(x^d) &\text{if } d\ge 2.\end{cases}\text{ as } x\to \infty.$$\end{theorem}

\begin{proof}We claim that for natural numbers $k,l$ the following statements are equivalent: $(i): kl\le x$ and $kl\equiv a (m)$ and $(ii): k\le x, l\le x/k, (k,m)|g$ and $l\equiv A\bar D (M)$, where $a=(k,m)A, k=(k,m)D, m=(k,m)M$ and $D\bar{D}\equiv 1 (M)$. So assume that there are $k,l\in\mathbb{N}$ such that $kl\le x$ and $kl\equiv a (m)$. Since $l\ge 1$ we find $k\le x$ and $l\le x/k$ is equivalent to $kl\le x$. Since $(k,m)$ divides both $k$ and $m$, it also has to divide $a$ and hence also $(a,m)=g$. Dividing $kl-a = cm$ through $(k,m)$ we find $Dl\equiv A (M)$. Since $(D,M)=1$ there is $\bar D$ such that $D\bar D \equiv 1 (M)$. If we multiply the congruence by $\bar D$ we get the congruence of $(ii)$. For the other direction we multiply the congruence by $D$ and find $Dl-A = cM$ for some integer $c$. Multiplying this equation by $(k,m)$ we find that $kl\equiv a\mod m$.\\
Recall that $J_d=\mu\star I_d$. Hence we find $$\sum_{\substack{n\le x\\ n\equiv a (m)}} J_d(n)=\sum_{\substack{n\le x\\ n\equiv a (m)}} \sum_{k|n}\mu(k)(n/k)^d=\sum_{\substack{kl\le x\\ kl\equiv a (m)}} \mu(k)l^d=\sum_{\substack{k\le x\\ (k,m)|g}} \mu(k) \sum_{\substack{l\le x/k\\l\equiv A\bar{D}(M)}} l^d,$$ where $A, \bar D$ and $M$ are defined as in the claim.
By Lemma \ref{lem:sumnk} the inner sum equals $\frac{x^{d+1}}{(d+1)k^{d+1}M}+O_{d,M}(x^d/k^d)=\frac{x^{d+1} (k,m)}{(d+1)k^{d+1}m}+O_{d,M}(x^d/k^d)$. Since $M$ is a divisor of $m$, we can find an absolute constant which holds for all $k$ and therefore the sum of the error terms is $O_{d,m}(x^d\sum_{k\le x}k^{-d})=O_{d,m}(x^d\log(x))$ where we need the log term only if $d=1$ since the sum is $O(\log(x))$ if $d=1$ and $O(1)$ if $d\ge 2$. Therefore we find $$\sum_{\substack{n\le x\\ n\equiv a (m)}} J_d(n)=\sum_{\substack{k\le x\\ (k,m)|g}} \mu(k)\frac{x^{k+1} (k,m)}{(d+1)k^{d+1}m} +R(x)= \frac{x^{d+1}}{(d+1)m}\sum_{g'|g}g'\sum_{\substack{k\le x\\ (k,m)=g'}} \frac{\mu(k)}{k^{d+1}} +R(x)$$ where $R(x)=O_m(x\log(x))$ if $d=1$ and $R(x)=O_{d,m}(x^d)$ if $d\ge 2$.\\
Observe that for all $g'|m$ we have $\{k\in\mathbb{N}: k\le x, (k,m)=g'\}=\{g'D : D\le x/g', (D, m/g')=1\}$. This allows us to write $$\sum_{\substack{k\le x\\ (k,m)=g'}} \frac{\mu(k)}{k^{d+1}}=\frac{1}{g'^{d+1}}\sum_{\substack{D\le x/g'\\ (D,m/g')=1}} \frac{\mu(g'D)}{D^{d+1}}=\frac{1}{g'^{d+1}}\sum_{\substack{D\le x/g'\\ (D,m/g')=1\\(D,g')=1}} \frac{\mu(g'D)}{D^{d+1}}=\frac{\mu(g')}{g'^{d+1}}\sum_{\substack{D\le x/g'\\ (D,m)=1}} \frac{\mu(D)}{D^{k+1}}$$ since $\mu(mn)=0$ if $(m,n)>1$. By Lemma \ref{lem:zeta} the sum is $\frac{m^{d+1}}{\zeta(d+1)J_{d+1}(m)}+O_d(g'^d/x^d).$ Hence $$\sum_{\substack{k\le x\\ (k,m)=g'}} \frac{\mu(k)}{k^{d+1}}= \frac{\mu(g')m^{d+1}}{\zeta(d+1)g'^{d+1}J_{d+1}(m)}+O_d(1).$$
Finally we plug this into the equation above and find \begin{align*}\sum_{\substack{n\le x\\ n\equiv a (m)}} J_d(n)&=\frac{m^dx^{d+1}}{(d+1)\zeta(d+1)J_{d+1}(m)}\sum_{g'|g}\frac{\mu(g')}{g'^d} + O_d(g)+R(x)\\
&=\frac{m^dJ_d(g)x^{d+1}}{(d+1)\zeta(d+1)g^dJ_{d+1}(m)} + O_{d,m}(1)+R(x)\end{align*}
since $J_d(g)=\sum_{g'|g}\mu(g')(g/g')^d$. The error term is $O_{m}(x\log(x))$ if $d=1$ and $O_{d,m}(x^d)$ if $d\ge 2$.
\end{proof}

\begin{lemma}\label{lem:conv}Let $d,n\in\mathbb{N}$. Then $\sum_{a=1}^n (a,n)J_d((a,n))= J_{d+1}(n)$.\end{lemma}
\begin{proof}Fix $g$ a divisor of $n$. If $1\le a\le n$ satisfies $(a,n)=g$, then $a=bg$ for some $b\le n/g$ coprime to $n/g$ and hence the number of such $a$ equals $\varphi(n/g)$. Thus we can write the sum in question as $\sum_{g|n}\varphi(n/g)gJ_d(g)$, which is nothing else than $((I_1 J_d)\star \varphi)(n)$. Thus we want to prove $(I_1J_d)\star\varphi = J_{d+1}$. Since $I_l$ is completely multiplicative we have \begin{align*}(I_1J_d)\star\varphi&= (I_1(I_d\star \mu))\star (I_1\star\mu)=I_{d+1}\star (I_1\mu)\star (I_1\epsilon)\star \mu\\
&=I_{d+1}\star\mu\star (I_1(\mu\star\epsilon))=J_{d+1}\star (I_1\eta) =J_{d+1}\star\eta = J_{d+1}.\qedhere\end{align*}\end{proof}

\begin{lemma}\label{lem:numberOfFixedOrder}Let $K$ be an algebraically closed field of characteristic $p$ and $d,n\in\mathbb{N}$. Then the number of $\boldsymbol\zeta\in \mathbb{G}_m^d(K)$ such that $\mathrm{ord}(\boldsymbol\zeta)=n$ is $J_d(n)$ if $(p\nmid n$ or $p=0)$ and is $0$ if $p|n$ and $p>0$.\end{lemma}
\begin{proof}Note that for all $\boldsymbol\zeta\in\mathbb G_m^d$ we have $\boldsymbol\zeta^n=\mathbf 1$ if and only if $\mathrm{ord}(\boldsymbol\zeta)|n$. Let $A_{d}(n)$ be the number of elements in $\mathbb{G}_m^d$ of order $n$. The number of solutions $\boldsymbol\zeta\in\mathbb G_m^d$ of $\boldsymbol\zeta^n=\mathbf 1$ equals $(\#\mu_n)^d$. If $p=0$ we have $\#\mu_n=n$ and if $p>0$ and $n=p^\alpha m,p\nmid m$ we have $\zeta^n-1=(\zeta^m-1)^{p^\alpha}$ and thus $\zeta^n=1$ if and only if $\zeta^m=1$. Since $\zeta^m=1$ and its formal derivative $m\zeta^{m-1}$ have no common roots, the equation $\zeta^m=1$ has $m$ distinct roots. Therefore $\#\mu_{p^\alpha m}=m$. Let $B_{d}(n)=(\#\mu_n)^d$. Then counting solutions in two ways yields $B_{d}(n)=\sum_{m|n}A_{d}(m).$ So by M\"obius inversion we have $A_{d}(n)=\sum_{m|n}\mu(m)B_{d}(n/m)$.
If $p=0$ we have $B_d(n)=n^d=I_d(n)$ and therefore $A_d=\mu\star I_d=J_d$ and we are left to prove the lemma in positive characteristic.\\
So let us assume $p>0$ and write $n=p^\alpha m$ where $p\nmid m$. Since $p$ and $m$ are coprime we can factor any divisor of $n$ uniquely in a product of coprime divisors of $p^\alpha$ and $m$ respectively. Thus we find $$A_{d}(n)=\sum_{k|p^\alpha}\sum_{l|m}\mu(kl)B_{d}(n/(kl))=\left(\sum_{k|p^\alpha}\mu(k)\right)\left(\sum_{l|m}\mu(l)(m/l)^d\right)=\eta(p^\alpha)J_d(m)$$ using that $\mu\star \epsilon=\eta$. If $p|n$ then $p^\alpha\ne 1$ and hence $A_{d}(n)=0$ since $\eta(p^\alpha)=0$. And if $p\nmid n$ we have $p^\alpha=1, m=n$ and hence $A_{d}(n)=J_d(n)$.\end{proof}

\begin{definition}Let $\boldsymbol\zeta\in(\mathbb G_m^n)_{\mathrm{tors}}$ and $G<\mathbb G_m^n$ an algebraic subgroup. Then we define $$\mathrm{ord}(\boldsymbol\zeta G)=\min\{\mathrm{ord}(\boldsymbol\xi): \boldsymbol\xi \in \boldsymbol\zeta G \text{ of finite order}\}.$$\end{definition}

\begin{proof}[Proof of Theorem \ref{thm:countCoset}]Since a monoidal transform preserves orders, we can assume that the coset is of the form $\{\boldsymbol\eta\}\times\mathbb G_m^d$ for some $\boldsymbol\eta\in\mathbb G_m^{n-d}$ of order $g$ by Lemma \ref{lem:redTorsCoset}. Then the order of $(\boldsymbol{\eta,\zeta})$ equals $\mathrm{lcm}(g,\mathrm{ord}(\boldsymbol\zeta))=\frac{g}{(g,\mathrm{ord}(\boldsymbol\zeta))}\mathrm{ord}(\boldsymbol\zeta)$ for all $\boldsymbol\zeta\in\mathbb{G}_m^d$. Fix $1\le a\le g$ and assume that $\mathrm{ord}(\boldsymbol\zeta)\equiv a \mod g$ and $(\boldsymbol{\eta,\zeta})$ is of order bounded by $T$. Thus $g/(g,a)\mathrm{ord}(\boldsymbol\zeta)\le T$ and hence $\mathrm{ord}(\boldsymbol\zeta)\le T(g,a)/g$.\\
We first assume that $p=0$. Then by Lemma \ref{lem:numberOfFixedOrder} the number of elements of order $n$ in $\mathbb{G}_m^d$ is given by $J_d(n)$. We use Theorem \ref{thm:Jord} and Lemma \ref{lem:conv} to find \begin{align*}(\boldsymbol\eta G)_T&=\sum_{a=1}^g\sum_{\substack{h\le T(g,a)/g\\ h\equiv a (g)}} J_d(h)=\sum_{a=1}^g \frac{J_d((g,a))(g,a)T^{d+1}}{(d+1)\zeta(d+1)gJ_{d+1}(g)}+R(T)\\
&=\frac{T^{d+1}}{(d+1)\zeta(d+1)gJ_{d+1}(g)}\sum_{a=1}^g (g,a)J_k((g,a))+R(T)= \frac{T^{d+1}}{(d+1)\zeta(d+1)g} +R(T),\end{align*} where $R(T)=O_g(T\log(T))$ if $d=1$ and $R(T)=O_{d,g}(T^d)$ if $d\ge 2$. \\
For the case $p>0$ let us start with some remarks. Since $g$ is the order of $\boldsymbol\eta$ it is coprime to $p$. Thus there exists an integer $\bar p$ such that $p\bar p\equiv 1\; (\mathrm{mod } \;g)$. Then for natural numbers $h$ we have $h\equiv a \;(\mathrm{mod}\; g)$ and $h\equiv 0 \;(\mathrm{mod} \;p)$ if and only if $h\equiv ap\bar p\; (\mathrm{mod} \; gp)$. Let us also compute $(a\bar p p, gp)=p(a\bar p, g)=p(a,g)$ since $(\bar p,g)=1$. By Lemma \ref{lem:numberOfFixedOrder} there are also $J_d(n)$ torsion points of order $n$ if $(n,p)=1$. Otherwise there are none. Thus we have to subtract \begin{align*}\sum_{a=1}^g \sum_{\substack{h\le T(g,a)/g\\h\equiv a (g)\\h\equiv 0(p)}}J_k(h)&=\sum_{a=1}^g \sum_{\substack{h\le T(g,a)/g\\h\equiv ap\bar p (pg)}}J_d(h)=\sum_{a=1}^g \frac{J_d(p(g,a))(g,a)T^{d+1}}{(d+1)\zeta(d+1)gJ_{d+1}(gp)}+R(T)\\
&=\frac{J_d(p)T^{d+1}}{J_{d+1}(p)(d+1)\zeta(d+1)gJ_{d+1}(g)}\sum_{a=1}^g (g,a)J_d((g,a))+R(T)\\
&= \frac{J_{d}(p)T^{d+1}}{J_{d+1}(p)(d+1)\zeta(d+1)g} +R(T)\end{align*} from the term in characteristic $0$ above. Here $R(T)$ is as above and we used again Theorem \ref{thm:Jord} for the second and Lemma \ref{lem:conv} for the last equation.
Since $1-\frac{J_d(p)}{J_{d+1}(p)}=\frac{p^{d+1}-1-(p^d-1)}{p^{d+1}-1}=\frac{p-1}{p-p^{-d}}$ the claimed result follows.\end{proof}

\begin{remark}Let $S_{a,m,d,T}=\sum_{\substack{n\le T\\ n\equiv a(m)}}J_d(n)$. Then we have $$\lim_{T\to\infty} \frac{S_{a,m,d,T}}{S_{b,m,d,T}}=\frac{\prod_{p|(a,m)}(1-p^{-d})}{\prod_{p|(b,m)}(1-p^{-d})}.$$ In particular if we compare the contribution in congruence classes mod $p$, the proportions are $1:\dots:1:(1-p^{-d})$ and hence the proportion of all the coprime congruence classes is $(p-1)/(p-p^{-d})$, which coincides with the factor found in Theorem \ref{thm:countCoset}.\end{remark}

\section{Outside torsion cosets}\label{sec:out}
We count the torsion points which are not contained in a positive dimensional torsion coset. We use Minkowski's second theorem to construct a one dimensional subgroup which is defined by equations with small coefficients. In the first part we estimate the number of torsion points lying in such a subgroup, and then sum over all such groups. As usual $K$ always denotes an algebraically closed field.
\begin{definition}Let $n,m\in\mathbb{N}$ and $A\in\mathrm{Mat}_{n\times m}(\mathbb{Z})$. Then we call a factorisation $A=TDS$ Smith normal form of $A$ if $S\in\mathrm{GL}_n(\mathbb Z), T\in\mathrm{GL}_m(\mathbb{Z})$ and $D=\mathrm{diag}_{n\times m}(\boldsymbol\alpha)$ for some $\boldsymbol\alpha\in\mathbb{Z}_{\ge 0}^n$ such that $\alpha_{i+1}\mathbb Z\subset \alpha_{i}\mathbb Z$ for all $1\le i < \min\{n,m\}$.\end{definition}

\begin{remark}That every matrix $A\in\mathrm{Mat}_{n\times m}(\mathbb{Z})$ has a Smith normal form is a classical result established in \cite{SNF}.\end{remark}

\begin{definition}Let $P\in K[X_1^{\pm 1},\dots,X_n^{\pm 1}]\setminus\{0\}$ be a Laurent polynomial. There exists a unique polynomial $Q\in K[X_1,\dots X_n]$ which is coprime to $X_1\cdots X_n$ and $\mathbf v\in\mathbb Z^n$ such that $P=(X_1,\dots,X_n)^{\mathbf v}Q$. We define the Laurent degree as $\mathrm{Ldeg}(P)=\deg(Q)$.\end{definition}

\begin{remark}Since the zero locus of $P$ and $Q$ above in $\mathbb G_m^n$ are equal, a Laurent polynomial $P\in K[X^{\pm 1}]\setminus\{0\}$ of degree $d=\mathrm{Ldeg}(P)$ has at most $d$ roots.\end{remark}

\begin{lemma}\label{lem:Ldeg}Let $P\in K[X_1^{\pm 1},\dots, X_n^{\pm1 }]$, $\mathbf y\in\mathbb G_m^n$ and $\mathbf u\in\mathbb Z^n$. Suppose that $P\left(\mathbf y Y^{\mathbf u^\top}\right)\ne 0$. Then the Laurent degree of $P\left(\mathbf y Y^{\mathbf u^\top}\right)\in K[Y^{\pm 1}]$ is bounded by $2|\mathbf u|\mathrm{Ldeg}(P)$.\end{lemma}

\begin{proof}Let $R\in K[Y]\setminus\{0\}$, $s_0=\min \mathrm{Supp}(R)$ and $s_1=\max \mathrm{Supp}(R)$. Then we have $\mathrm{Ldeg}(R)=s_1-s_0$.
By definition there exists $\mathbf v\in\mathbb Z^n$ and $Q\in K[X_1,\dots, X_n]$ of degree $\mathrm{Ldeg}(P)$ such that $P=(X_1,\dots,X_n)^{\mathbf v}Q$. Let $R=P\left(\mathbf y Y^{\mathbf u^\top}\right)$ and suppose that $s\in\mathrm{Supp}(R)$. Then there exists $\mathbf i\in \mathrm{Supp}(P)$ such that $s=\mathbf{ u^\top i}$ and hence $\mathbf j\in\mathrm{Supp}(Q)$ such that $s=\mathbf{u^\top v+u^\top j}$. We can bound $|\mathbf{u^\top j}|\le \sum_{i=1}^n |u_i|j_i\le |\mathbf u|\sum_{i=1}^n j_i\le |\mathbf u|\deg(Q)$. Thus there exist $r_0, r_1\in\mathbb Z$ bounded by $|\mathbf u|\deg(Q)$ such that $s_i=\mathbf{u^\top v}+r_i, i=0,1$. Then we find $\mathrm{Ldeg}(R)=s_1-s_0=r_1-r_0\le 2|\mathbf u| \mathrm{Ldeg}(P)$.\end{proof}

\begin{definition}We call a set $X\subset\mathbb{G}_m^n$ algebraic over $K$ if there exists polynomials $P_1,\dots, P_m\in K[X_1^{\pm 1},\dots, X_n^{\pm 1}]$ such that $$X=\{\mathbf x\in\mathbb{G}_m^n: P_1(\mathbf x)=\dots=P_m(\mathbf x)=0\}.$$ Let $X\subset \mathbb{ G}_m^n$ be an algebraic subset of dimension $d$. We call it admissible if it is finite or $d\ge 1$ and it does not contain a torsion coset of dimension $d$.
We define $$X^*=X\setminus \bigcup_{\substack{C\subset X \text{ torsion coset}\\ \dim(C)\ge 1}}C.$$
For $P\in K[X_1^{\pm 1},\dots,X_n^{\pm 1}]$ we let $Z(P)=\{\mathbf x\in\mathbb{G}_m^n(K): P(\mathbf x)=0\}$ be the zero locus of $P$ in $\mathbb{G}_m^n.$\end{definition}

\begin{lemma}\label{lem:fixedA}Let $n\ge 2$ and $\mathbf a_1,\dots\mathbf a_{n-1}\in\mathbb{Z}^n$ be linearly independent and $P\in K[X_1^{\pm 1},\dots,X_n^{\pm 1}]\setminus\{0\}$ and $X=Z(P)$. Then $$\#\left(X^*\cap\bigcap_{i=1}^{n-1}\left\{\mathbf x\in\mathbb{G}_m^n: \mathbf x^{\mathbf a_i}=1\right\}\right)\le 2(n-1)!\mathrm{Ldeg}(P)|\mathbf a_1|\cdots |\mathbf a_{n-1}|.$$\end{lemma}
\begin{proof}A suitable version of Bézouts Theorem would suffice for our purposes. However we follow a more elementary and self-contained approach. Let $A=\begin{pmatrix} \mathbf a_{1},\dots, \mathbf a_{n-1}\end{pmatrix}\in\mathrm{Mat}_{n\times (n-1)}(\mathbb{Z})$. Let $A=SDT$ be a Smith normal form of $A$ with $D=\mathrm{diag}_{n\times (n-1)}(\alpha_1,\dots,\alpha_{n-1})$ and note that $\alpha_1,\dots,\alpha_{n-1}\ge 1$. Let $\mathbf u_1,\dots,\mathbf u_n\in\mathbb{Z}^n$ such that $S^{-1}=\begin{pmatrix} \mathbf u_{1}^\top\\ \vdots\\ \mathbf u_{n}^\top\end{pmatrix}$. We claim that $$\bigcap_{i=1}^{n-1}\left\{\mathbf x\in\mathbb{G}_m^n: \mathbf x^{\mathbf a_i}=1\right\}=\left\{y_1^{\mathbf u_{1}^\top}\cdots y_n^{\mathbf u_{n}^\top}: y_i^{\alpha_i}=1\text{ for }i=1,\dots, n-1\text{ and } y_n\in\mathbb{G}_m^n\right\}.$$ Note that $\mathbf x\in\mathbb{G}_m^n$ is in the intersection on the left if and only if $\mathbf x^A=\mathbf{1}$ and since $T$ has an inverse in $\mathrm{Mat}_{n-1}(\mathbb{Z})$ this is equivalent to $\mathbf x^{SD}=\mathbf 1$. Let $\mathbf{ y=x}^S$. Then $\mathbf y^D=\mathbf{1}$ if and only if $y_i^{\alpha_i}=1$ for $i=1,\dots,n-1$. There is no condition on $y_n$. Therefore any $\mathbf x\in\mathbb{G}_m^n$ such that $\mathbf x^{SD}=\mathbf 1$ is of the form $\mathbf{x=y}^{S^{-1}}=y_1^{\mathbf u_{1}^\top}\cdots y_n^{\mathbf u_{n}^\top}$ where $y_i^{\alpha_i}=1,i<n$ and $y_n\in\mathbb{G}_m.$
Note that $\mathbf u_n$ is primitive since it is a row of $S^{-1}\in\mathrm{GL}_n(\mathbb{Z})$. We also find $$\mathbf u_n^\top A=\mathbf e_n^\top S^{-1}A=\mathbf e_n^\top DT=\mathbf 0.$$ Since the kernel of $A^\top$ is one dimensional this shows that primitivity and being in the kernel of $A^\top$ determines $\mathbf u_n$ up to sign.\\
Let $g=d_{n-1}(A)$ and $v_i=\det(\mathbf e_i, A)/g\in\mathbb{Z}.$ Since $\gcd(v_1,\dots, v_n)=d_{n-1}(A)/g=1$, the vector $\mathbf v=(v_1,\dots, v_n)$ is primitive. Then for any $j=1,\dots, n-1$ we have $$\mathbf a_j^\top \mathbf v = a_{1,j}\det(\mathbf e_1, A)/g +\dots +a_{n,j}\det(\mathbf e_n, A)/g= \det(\mathbf a_j, A)/g=0$$ and hence $|\mathbf u_n|=|\mathbf v|\le (n-1)! |\mathbf a_1|\cdots |\mathbf a_{n-1}|/g$ by the Leibniz formula. Lemmas \ref{lem:diInv} and \ref{lem:diDiag} imply that $g=d_{n-1}(D)=\alpha_1\cdots\alpha_{n-1}$.\\
Now fix $(y_1,\dots,y_{n-1})\in\mu_{\alpha_1}\times\dots\times \mu_{\alpha_{n-1}}$. If $P\left(y_1^{\mathbf u_{1}^\top}\cdots y_{n-1}^{\mathbf u_{n-1}^\top}Y^{\mathbf u_{n}^\top}\right)= 0$, then for all $y_n\in\mathbb{G}_m$ the element $(y_1^{\mathbf u_{1}^\top}\cdots y_n^{\mathbf u_{n}^\top})$ is contained in a torsion coset of positive dimension and hence $\{y_1^{\mathbf u_{1}^\top}\cdots y_n^{\mathbf u_{n}^\top}: y_n\in\mathbb{G}_m\}\cap X^*=\emptyset$.\\
In the other case an element $\mathbf x\in \{y_1^{\mathbf u_{1}^\top}\cdots y_n^{\mathbf u_{n}^\top}: y_n\in\mathbb{G}_m\}\cap X$ is of the form  $y_1^{\mathbf u_{1}^\top}\cdots y_n^{\mathbf u_{n}^\top}$ for some root $y_n\in\mathbb{G}_m$ of the Laurent polynomial $Q(Y)=P\left(y_1^{\mathbf u_{1}^\top}\cdots y_{n-1}^{\mathbf u_{n-1}^\top}Y^{\mathbf u_{n}^\top}\right)$. By Lemma \ref{lem:Ldeg} we have $\mathrm{Ldeg}(Q)\le 2|\mathbf u_n|\mathrm{Ldeg}(P)$. Since $|\mathbf u_n|\le (n-1)! |\mathbf a_1|\cdots |\mathbf a_{n-1}|/g$ we have $$\#\left(\left\{y_1^{\mathbf u_{1}^\top}\cdots y_n^{\mathbf u_{n}^\top}: y_n\in\mathbb{G}_m^n\right\}\cap X\right)\le 2(n-1)!\mathrm{Ldeg}(P)|\mathbf a_1|\cdots |\mathbf a_{n-1}|/g.$$ Summing over all $(y_1,\dots,y_{n-1})\in\mu_{\alpha_1}\times\dots\times \mu_{\alpha_{n-1}}$ - which are at most $g$ distinct elements - we get the claimed bound.\end{proof}

\begin{lemma}\label{lem:sum1}Let $n\ge 2, \alpha_1,\dots ,\alpha_{n-1}\in [0,1),\beta\in(-1,\infty)$ and $A=\prod_{i=1}^{n-1}(1-\alpha_i)$. Then for all $T\ge 1$ we have $$\sum_{k_1\le T^{\alpha_1}}\cdots \sum_{k_i\le (T/(k_1\cdots k_{i-1}))^{\alpha_i}}\cdots \sum_{k_{n-1}\le (T/(k_1\cdots k_{n-2}))^{\alpha_{n-1}}} (k_1\cdots k_{n-1})^\beta\ll_{\boldsymbol\alpha,\beta} T^{(\beta+1)(1-A)}. $$\end{lemma}
\begin{proof}Comparing the sum with the integral yields that $\sum_{k\le x}k^\alpha \ll_\alpha x^{\alpha+1}$ as $x\to \infty$ for all real $\alpha >-1$. The proof is by induction on $n$. If $n=2$ we have $$\sum_{k\le T^{\alpha}} k^\beta\ll_\beta T^{\alpha(\beta+1)}=T^{(\beta+1)(1-(1-\alpha))}.$$So assume that $n\ge 3$ and the lemma holds for all $m\le n-1$. Let $A'=\prod_{i=2}^{n-1} (1-\alpha_i)$. Then \begin{align*}&\sum_{k_1\le T^{\alpha_1}}\cdots \sum_{k_i\le (T/(k_1\cdots k_{i-1}))^{\alpha_i}}\cdots \sum_{k_{n-1}\le (T/(k_1\cdots k_{n-2}))^{\alpha_{n-1}}} (k_1\cdots k_{n-1})^\beta\\
&\qquad =\sum_{k_1\le T^{\alpha_1}}k_1^\beta \sum_{k_2\le (T/k_1)^{\alpha_i}}\cdots \sum_{k_{n-1}\le ((T/k_1)/(k_2\cdots k_{n-2}))^{\alpha_{n-1}}} (k_2\cdots k_{n-1})^\beta\\
&\qquad \ll_\beta \sum_{k_1\le T^{\alpha_1}}k_1^\beta (T/k_1)^{(\beta+1)(1-A')}=T^{(\beta+1)(1-A')}\sum_{k_1\le T^{\alpha_1}}k_1^{\beta-(\beta+1)(1-A')}\end{align*} by induction since $T/k_1\ge T^{1-\alpha_1}\ge 1$.\\
We claim that the exponent $\beta-(\beta+1)(1-A')>-1$. This is equivalent to $\beta+1>(\beta+1)(1-A')$ and since $\beta>-1$ also to $A'>0$ which is true since the $\alpha_i$ are in $[0,1)$. Thus we can bound $$\sum_{k_1\le T^{\alpha_1}}k_1^{\beta-(\beta+1)(1-A')}\ll_{\boldsymbol\alpha,\beta} T^{\alpha_1(\beta+1-(\beta+1)(1-A'))}=T^{\alpha_1(\beta+1)A'}.$$
So the exponent in the bound equals $$(\beta+1)(1-A')+(\beta+1)\alpha_1A'=(\beta+1)(1-A'(1-\alpha_1))=(\beta+1)(1-A).\qedhere$$\end{proof}

\begin{definition}Let $F_n=[-1,1]^n$ and $\Lambda\subset \mathbb{Z}^n$ a subgroup of rank $n$. Then the successive minima of $\Lambda$ with respect to $F_n$ are defined as \begin{align*}\lambda_i(\Lambda)&=\inf\{\lambda\in \mathbb{R}_{\ge 0}: \lambda F_n\text{ contains } i \text{ linear independent lattice points of }\Lambda \}\\
&=\inf\{\lambda\in\mathbb{R}_{\ge 0}: \text{ there exist }i \text{ linear indepedent points of }\Lambda\text{ of norm at most }\lambda\}\end{align*} where the norm of a vector $\mathbf v\in\mathbb{Z}^n$ is given by $|\mathbf v|$.\end{definition}
\begin{lemma}\label{lem:Mink2}Let $n\in\mathbb N$ and $\Lambda\subset \mathbb{Z}^n$ be a subgroup of rank $n$. Then $\lambda_1(\Lambda)\cdots\lambda_n(\Lambda)\le [\mathbb Z^n:\Lambda]$.\end{lemma}
\begin{proof}This follows from Minkowski's second theorem (Theorem V on page 218 in \cite{M2}) and the fact that the volume of $F_n$ equals $2^n$. Furter the determinant of $\Lambda$ is given by the absolute value of the determinant of a matrix $B$ whose columns form a basis of $\Lambda$. Since $\det(\Lambda)=|\det(B)|=[\mathbb Z^n:B\mathbb Z^n]=[\mathbb Z^n: \Lambda]$ the determinant is given by $[\mathbb Z^n:\Lambda]$.\end{proof}

\begin{definition}For $\boldsymbol\zeta\in(\mathbb{G}_m^n)_{\mathrm{tors}}$ we define $\Lambda_{\boldsymbol\zeta}=\left\{\mathbf a\in\mathbb{Z}^n:\boldsymbol\zeta^{\mathbf a}=1\right\}$.\end{definition}

\begin{lemma}\label{lem:lambdaZeta}Let $\boldsymbol\zeta\in(\mathbb{G}_m^n)_{\mathrm{tors}}$ be of  order $N$. Then $\Lambda_{\boldsymbol\zeta}$ is a subgroup of $\mathbb Z^n$ of rank $n$ and $[\mathbb Z^n:\Lambda_{\boldsymbol\zeta}]=N$.\end{lemma}
\begin{proof}Since $\Lambda_{\boldsymbol\zeta}$ contains $N\mathbb Z^n$, it has full rank. Consider the group homomorphism $\phi: \mathbb{Z}^n\mapsto \mathbb{G}_m$, given by $\phi(\mathbf a)=\boldsymbol\zeta^{\mathbf a}$. Then $\ker(\phi)=\Lambda_{\boldsymbol\zeta}$ and $\mathrm{Im}(\phi)=\mu_N$. By the isomorphism theorem we find $[\mathbb Z^n:\Lambda_{\boldsymbol\zeta}]=\#\mu_N=N.$ The last equality holds clearly in characteristic zero. It is also true in positive characteristic $p>0$ since the order of an element in $(\mathbb G_m^n)_{\mathrm{tors}}$ is always coprime to $p$.\end{proof}

\begin{definition}For a set $V\subset \mathbb{G}_m^n$ and $T\ge 1$ we define $$V_T=\{\boldsymbol\zeta\in V\cap(\mathbb G_m^n)_{\mathrm{tors}} : \mathrm{ord}(\boldsymbol\zeta)\le T\}.$$\end{definition}

\begin{definition}A hypersurface $X\subset\mathbb G_m^n$ is the zero set $Z(P)$ of $P\in K[X_1^{\pm 1},\dots, X_n^{\pm 1}]\setminus (K[X_1^{\pm 1},\dots, X_n^{\pm 1}]^*\cup\{0\})$ in $\mathbb G_m^n$.\end{definition}
\begin{remark}Let $P\in K[X_1^{\pm 1},\dots, X_n^{\pm 1}]\setminus (K[X_1^{\pm 1},\dots, X_n^{\pm 1}]^*\cup\{0\}).$ Then the dimension of $Z(P)$ equals $n-1$.\end{remark}

\begin{lemma}\label{lem:D0}Let $X=Z(P)$ be a hypersurface in $\mathbb G_m^n$. Then $$\#X^*_T\ll_n \mathrm{Ldeg}(P)T^{n-1/n}.$$\end{lemma}
\begin{proof}Let $\boldsymbol\zeta\in\mathbb{G}_m^n$ be of order $N\le T$. Since there are only finitely many points in $\mathbb Z^n$ of finite norm, the infimum in the definition of $\lambda_i$ is a minimum. Thus there exists a system of $n$ linear independent vectors $\mathbf a_1,\dots,\mathbf a_n\in \Lambda_{\boldsymbol\zeta}$ such that $|\mathbf a_i|=\lambda_i(\Lambda_{\boldsymbol\zeta})$ for all $i=1,\dots, n$. Then by Lemmas \ref{lem:Mink2} and \ref{lem:lambdaZeta} we have $|\mathbf a_1|\cdots |\mathbf a_n|\le [\mathbb Z^n:\Lambda_{\boldsymbol{\zeta}}]=N\le T$. Since $|\mathbf a_1|\le \dots \le |\mathbf a_n|$ we have $|\mathbf a_i|^{n-(i-1)}\le T/(|\mathbf a_1|\cdots |\mathbf a_{i-1}|)$ for all $i=1,\dots, n$ which yields $|\mathbf a_i|\le \left(T/(|\mathbf a_1|\cdots |\mathbf a_{i-1}|)\right)^{1/(n-(i-1))}$. Therefore \begin{align*}(\mathbb{G}_m^n)_T&\subset \{\boldsymbol\zeta\in\mathbb{G}_m^n: \text{ there exist linearly independent }\mathbf a_1,\dots ,\mathbf a_{n}\text{ such that }\boldsymbol\zeta^{\mathbf a_i}=1, \\
&\quad\quad\text{ and }|\mathbf a_i|\le \left(T/(|\mathbf a_1|\cdots |\mathbf a_{i-1}|)\right)^{1/(n-i+1)} \text{ for all }i=1,\dots, n\}.\end{align*} The number of vectors $\mathbf a\in\mathbb{Z}^n$ satisfying $|\mathbf a|=k$ is bounded by a constant depending on $n$ times $k^{n-1}$ . By Lemma \ref{lem:fixedA} have \begin{align*}\#X^*_T&\le \sum_{\mathbf a_1,\dots ,\mathbf a_{n-1}}\#\left(X^*\cap\bigcap_{i=1}^{n-1}\left\{\mathbf x\in\mathbb{G}_m^n: \mathbf x^{\mathbf a_i}=1\right\}\right)\\
&\ll_{n} \mathrm{Ldeg}(P)\sum_{\mathbf a_1,\dots ,\mathbf a_{n-1}}|\mathbf a_1|\cdots |\mathbf a_{n-1}|\\
&\ll_n \mathrm{Ldeg}(P)\sum_{k_1,\dots ,k_{n-1}}(k_1\cdots k_{n-1})^n;\end{align*}
in the first two sums the $\mathbf a_i$ satisfy $|\mathbf a_i|\le \left(T/(|\mathbf a_1|\cdots |\mathbf a_{i-1}|)\right)^{1/(n-i+1)}$, in the last step we sum over $k_i=|\mathbf a_i|$ such that $k_i\le \left(T/(k_1\cdots k_{i-1})\right)^{1/(n-i+1)}$. The last sum can be bounded by Lemma \ref{lem:sum1} with $\alpha_i = 1/(n-i+1)$ and $\beta=n$. We have $\prod_{i=1}^{n-1}(1-\alpha_i)=\prod_{i=2}^n\frac{i-1}{i}=1/n$ and thus the exponent equals $(n+1)(1-1/n)=n-1/n$.\end{proof}

\section{Common roots and admissibility}\label{sec:admiss}
This section contains technical statements on algebraic subgroups of $\mathbb G_m^n$. As we work in arbitrary characteristic and as the literature (cf. \cite{BG}) concerns mainly characteristic zero we give full proofs.

\begin{lemma}\label{lem:xU}Let $m,n\in\mathbb{N}$, $U\in\mathrm{Mat}_{n\times m}(\mathbb Z)$ of rank $r$ and $\boldsymbol\zeta\in(\mathbb{G}_m^m)_{\mathrm{tors}}$ be a torsion point. Then the set of $\mathbf x\in\mathbb{G}_m^n$ such that $\mathbf x^U=\boldsymbol\zeta$ is either empty or a finite union of torsion cosets of dimension $n-r$. Moreover the union is not empty if $r=m$.\end{lemma}
\begin{proof}Suppose that there exists $\mathbf y\in\mathbb G_m^n$ satisfies $\mathbf y^U=\boldsymbol{\zeta}$. Let $\Lambda=U\mathbb Z^m$. Then $\mathbf x^U=\boldsymbol\zeta$ if and only if $\mathbf x\in\mathbf y H_{\Lambda}$ and we can conclude with Lemma \ref{lem:HLambda}. For the second statement let $U=SDT$ be a Smith normal form of $U$, where $D=\mathrm{diag}_{n\times m}(\alpha_1,\dots,\alpha_m)$. Let $a=\alpha_1\cdots \alpha_m$ and $\beta_i=a/\alpha_i, i=1,\dots, m$. Let $V=T^{-1}\mathrm{diag}_{m\times n}(\beta_1,\dots, \beta_m) S^{-1}\in\mathrm{Mat}_{m\times n}(\mathbb Z)$. Then $VU=aE_m$. Since $\mathbb G_m$ is algebraically closed we can find $\boldsymbol\eta\in\mathbb G_m^n$ such that $\boldsymbol \eta^a=\boldsymbol\zeta$. Then $\mathbf y=\boldsymbol\eta^V$ satisfies $\mathbf y^U = \boldsymbol\eta^a =\boldsymbol\zeta.$\end{proof}

\begin{definition}We call $\mathbf u\in \mathbb Z^n\setminus \{0\}$ primitive if its entries are coprime.\end{definition}

\begin{lemma}\label{lem:irred}Let $\mathbf u\in\mathbb Z^n$ and $z\in\mathbb G_m$. Then $(X_1,\dots, X_n)^{\mathbf u}-z\in K[X_1^{\pm 1},\dots, X_n^{\pm 1}]$ is irreducible if and only if $\mathbf u$ is primitive.\end{lemma}
\begin{proof}Note that for all $A\in\mathrm{GL}_n(\mathbb Z)$ the map $\phi_A: \sum_{\mathbf i\in\mathbb Z^n}p_{\mathbf i}(X_1,\dots,X_n)^{\mathbf i}\mapsto \sum_{\mathbf i\in\mathbb Z^n}p_{\mathbf i}(X_1,\dots,X_n)^{A\mathbf i}$ is an automorphism of $K[X_1^{\pm 1},\dots X_n^{\pm 1}]$. By the theorem of elementary divisors (Theorem III.7.8 in \cite{Lang}) there exists $\lambda\in\mathbb Z\setminus\{0\}$ and a basis $\mathbf b_1,\dots, \mathbf b_n$ of $\mathbb Z^n$ such that $\mathbf u=\lambda \mathbf b_1$. Let $B=(\mathbf b_1,\dots, \mathbf b_n)$. Then $\mathbf {X^u}-z$ is irreducible if and only if $\phi_{B^{-1}}(\mathbf {X^u}-z)=\mathbf X^{\lambda B^{-1}\mathbf b_1}-z=X_1^\lambda-z$ is irreducible too. It is not hard to show that this is the case if and only if $1=|\lambda|=\gcd(\mathbf u)$.\end{proof}

\begin{lemma}\label{lem:R}Let $R_1,\dots, R_m\in K[X_1^{\pm 1},\dots, X_n^{\pm 1}]\setminus\{0\}$ each be a product of factors of the form $\mathbf{X^u}-\zeta$, where $\zeta$ is a root of unity and $\mathbf u\in\mathbb Z^n$ is primitive. Assume that $m\ge 2$ and $R_1,\dots, R_m$ are coprime. Then the set of their common roots is a finite union of torsion cosets of dimension at most $n-2$.\end{lemma}
\begin{proof}If there are empty products there are no common roots, therefore we can assume that there are no $R_i$ which equal $1$. Suppose that $\mathbf x\in\mathbb{G}_m^n$ is a common root. Then for each $i\in \{1,\dots m\}$ there exist a primitive $\mathbf u_i\in\mathbb{Z}^n$ and $\zeta_i$ a root of unity such that $\mathbf X^{\mathbf u_i}-\zeta_i | R_i$ and $\mathbf x^{\mathbf u_i} =\zeta_i$. Let $U=(\mathbf u_1,\dots,\mathbf u_m)\in\mathrm{Mat}_{n\times m}(\mathbb{Z})$ and $\boldsymbol\zeta=(\zeta_1,\dots,\zeta_m)$. Then $\mathbf x^U=\boldsymbol\zeta$. Let $U=SDT$ be a smith normal form of $U$. The rank $j$ of $U$ is the largest index such that $\alpha_j\neq 0$. Since the $\mathbf u_i$ are primitive we have $U\ne\mathbf 0$ and therefore there exists such a $j$. Assume that $j=1$. Then we find $$\mathbf u_i= U\mathbf e_i = SDT\mathbf e_i=SD\mathbf t_i=\alpha_1t_{1,i}S\mathbf e_1=\alpha_1 t_{1,i}\mathbf{s}_1.$$ Since the $\mathbf u_i$ are primitive, we have $\alpha_1t_{1,i}=\pm1$. Thus $\mathbf u_i=\pm \mathbf u_1$. So for all $i$ there is some root of unity $\xi_i$ such that $\mathbf X^{\mathbf u_1}-\xi_i |R_i$ and $\mathbf x^{\mathbf u_1}=\xi_i $, since $\mathbf {X^u}-\zeta = -\zeta\mathbf{X^u}(\mathbf{X^{-u}}-\zeta^{-1})$. Since the $R_i$ are coprime, not all $\xi_i$ are the same and hence there is no common root. So we can assume that $j\ge 2$. Then the set of $\mathbf x\in\mathbb{G}_m^n$ such that $\mathbf x^U=\boldsymbol\zeta$ is a finite union of torsion cosets of dimension $n-j\le n-2$ by Lemma \ref{lem:xU}. Since $U$ is from a finite set, this shows the claim.\end{proof}

\begin{lemma}\label{lem:divide}Let $\mathbf{u}\in\mathbb{Z}^n\setminus\{0\}, z\in \mathbb{G}_m$ and $P\in K[X_1^{\pm 1},\dots, X_n^{\pm 1}]\setminus\{0\}$. Assume that $\mathbf {X^u}-z |P$. Then there exist $\mathbf {i,j}\in \mathrm{Supp}(P)$ and $k\in\mathbb{N}$ such that $k\mathbf{ u=i-j}.$\end{lemma}
\begin{proof}Let $Q\in K[X_1^{\pm 1},\dots, X_n^{\pm 1}]\setminus\{0\}$ be the Laurent polynomial such that $P=(\mathbf {X^u}-z)Q$. Let $p_{\mathbf i}$ and $q_{\mathbf i}$ be the coefficients of $P$ and $Q$. Then we have $p_{\mathbf i}=q_{\mathbf {i-u}}-zq_{\mathbf i}$ for all $\mathbf i\in\mathbb{Z}^n$. Fix any $\mathbf i_0\in \mathrm{Supp}(Q)$ and for all $k\in\mathbb{Z}$ let $\mathbf i_k=\mathbf i_0+k\mathbf u$. Let $l=\max\{k\in\mathbb{Z}: \mathbf{i}_k\in\mathrm{Supp}(Q)\}$ and $\mathbf{ i=i}_{l+1}$. Then $p_{\mathbf i}=q_{\mathbf{i-u}}-zq_{\mathbf i}=q_{\mathbf{i-u}}\neq 0$ and therefore $\mathbf i\in \mathrm{Supp}(P)$. Similarly let $l'=\min\{k\in\mathbb{Z}: \mathbf{i}_k\in\mathrm{Supp}(Q)\}$ and $\mathbf{ j=i}_{l'}$. Then $p_{\mathbf j}=q_{\mathbf j-\mathbf u}-zq_{\mathbf j}=-zq_{\mathbf j}\neq 0$ and therefore $\mathbf j\in \mathrm{Supp}(P)$. Then $\mathbf{i-j}=\mathbf{i}_{l+1}-\mathbf i_{l'}= (l-l'+1)\mathbf u$. Since $l\ge l'$ the scalar $k=l-l'+1$ is a natural number.\end{proof}

\begin{lemma}\label{lem:artin}Let $G<\mathbb G_m^n$ be a subgroup of codimension $r$, $L\in\mathrm{Mat}_{n\times r}(\mathbb Z)$ a basis of $\Lambda_G$ and $P=\sum_{\mathbf i\in\mathbb Z^n}p_{\mathbf i}(X_1,\dots, X_n)^{\mathbf i}\in K[X_1^{\pm 1},\dots, X_n^{\pm 1}]$. For $\mathbf a\in\mathbb Z^n$ the map $\chi_{\mathbf a}: G\to \mathbb G_m^n, \mathbf x\mapsto \mathbf {x^a}$ is a character on $G$. For a character $\chi: G\to \mathbb G_m$ we put $L_\chi=\{\mathbf i\in \mathrm{Supp}(P): \chi_{\mathbf i}=\chi\}$. For those $\chi$ such that $L_{\chi}\ne \emptyset $ we fix $\mathbf i_\chi\in L_\chi$. For all $\mathbf i\in\mathrm{Supp}(P)$ we have $\mathbf {i-i}_{\chi_{\mathbf i}}\in \Lambda_G$. Thus there is a unique $\mathbf{ u_i}\in \mathbb G_m^r$ such that $\mathbf {i=i}_{\chi_{\mathbf i}}+L\mathbf {u_i}$. Finally we define $P_\chi= \sum_{\mathbf i\in L_\chi}p_{\mathbf i}(Y_1,\dots, Y_r)^{\mathbf{ u_i}}\in K[Y_1^{\pm 1}, \dots, Y_r^{\pm 1}]$.\\
Then we have $P=\sum_{\chi:L_\chi\ne \emptyset}(X_1,\dots, X_n)^{\mathbf i_\chi}P_\chi((X_1,\dots,X_n)^L)$ and if $\mathbf x\in\mathbb G_m^n$ is such that $P(\mathbf x G)=0$, then $\sum_{\mathbf i\in L_\chi}p_{\mathbf i}\mathbf x^{\mathbf i}=0=P_\chi(\mathbf x^L)$ for all characters $\chi$ such that $L_\chi\ne \emptyset$.\end{lemma}
\begin{proof}The argument has similarities with the proof of Proposition 3.2.14 in \cite{BG}. First we show that the $\mathbf{u_i}$ are well defined. Let $\mathbf i\in\mathrm{Supp}(P)$. Since $\mathbf i_{\chi_{\mathbf i}}\in L_{\chi_{\mathbf i}}$ we have $\mathbf x^{\mathbf i_{\chi_\mathbf i}}= \mathbf{ x^i}$ for all $\mathbf x\in G$ and hence $\mathbf i-\mathbf i_{\chi_\mathbf i}\in \Lambda_G$.

Let us abbreviate $\mathbf X=(X_1,\dots, X_n)$. Then we have $$P=\sum_{\mathbf i\in\mathrm{Supp}(P)}p_{\mathbf i}\mathbf{ X^ i}=\sum_{\chi: L_\chi\ne \emptyset}\sum_{\mathbf i\in L_\chi}p_{\mathbf i}\mathbf X^ {\mathbf i_{\chi_{\mathbf i}}+L\mathbf u_i}=\sum_{\chi: L_\chi\ne \emptyset}\mathbf X^{\mathbf i_\chi}\sum_{\mathbf i\in L_\chi}p_{\mathbf i}(\mathbf X^L)^{\mathbf u_i}=\sum_{\chi:L_\chi\ne \emptyset}\mathbf X^{\mathbf i_\chi}P_\chi(\mathbf X^L).$$

Suppose that $\mathbf x\in \mathbb G_m^n$ is such that $P(\mathbf xG)=0$. Let $\mathbf g\in G$. Then we have $$0=P(\mathbf {xg})=\sum_{\chi:L_\chi\ne \emptyset}\sum_{\mathbf i\in L_\chi}p_{\mathbf i}\mathbf{ x^i}\chi_{\mathbf i}(\mathbf g)=\sum_{\chi:L_\chi\ne \emptyset}\left(\sum_{\mathbf i\in L_\chi}p_{\mathbf i}\mathbf x^{\mathbf i}\right)\chi(\mathbf g).$$ By Artin's Lemma (Theorem 4.1 in chapter VI in \cite{Lang}) we must have $\sum_{\mathbf i\in L_\chi}p_{\mathbf i}\mathbf x^{\mathbf i}=0$ for all $\chi$.

We compute $$P_\chi(\mathbf x^L)=\sum_{\mathbf i\in L_\chi}p_{\mathbf i}\mathbf x^{L\mathbf{ u_i}}=\sum_{\mathbf i\in L_\chi}p_{\mathbf i}\mathbf x^{\mathbf{i-i}_{\chi_i}}=\mathbf x^{-\mathbf i_\chi}\left(\sum_{\mathbf i\in L_\chi}p_{\mathbf i}\mathbf x^{\mathbf i}\right)=0.$$\end{proof}

\begin{lemma}\label{lem:admP}Let $n\ge 2$, $P\in K[X_1^{\pm 1},\dots, X_n^{\pm 1}]$ and $X=Z(P)$. Then the following are equivalent \begin{enumerate}\item $X$ is admissible.
\item There are no primitive $\mathbf u\in\mathbb{Z}^n\setminus\{0\}$ and root of unity $\zeta\in\mathbb{G}_m$ such that $\mathbf {X^u}-\zeta | P$.
\item There are no $\mathbf u\in\mathbb{Z}^n\setminus\{0\}$ and root of unity $\zeta\in\mathbb{G}_m$ such that $\mathbf {X^u}-\zeta | P$.
\end{enumerate}\end{lemma}
\begin{proof}If $P = 0$, $X=\mathbb G_m^n$ is not admissible and every $Q\in K[X_1^{\pm 1},\dots ,X_n^{\pm 1}]$ divides $P$. And if $P\in K[X_1^{\pm1 },\dots, X_n^{\pm 1}]^*$ we have $X=\emptyset$ which is admissible. Since $\mathbf {X^u}-\zeta$ is never a unit if $\mathbf u\in\mathbb Z^n\setminus\{0\}$ and $\zeta\in \mathbb G_m^n$, the second and the third statement are also true. Thus we can assume that $P\not\in K[X_1,\dots, X_n]^*\cup \{0\}$ and hence the dimension of $X$ is $n-1\ge 1.$
The implication $1) \Rightarrow 2)$ follows directly from Lemma \ref{lem:xU} by contraposition.

$2)\Rightarrow 3):$ We prove the contraposition and assume that there exists $\boldsymbol\lambda\in\mathbb Z^n\setminus\{0\}$ such that $\mathbf{X}^{\boldsymbol\lambda}-\zeta| P$. Write $\boldsymbol\lambda =g \mathbf u$ for a primitive $\mathbf u\in\mathbb Z^n$ and a $g\ge 1$. Since $X-Y|X^g-Y^g\in\mathbb Z[X,Y]$ and we can write $\zeta=\eta^g$ for some $\eta\in\mathbb G_m$ we find $\mathbf {X^u}-\eta | \mathbf X^{\boldsymbol\lambda}-\zeta$ and hence there exists a primitive $\mathbf u\in\mathbb{Z}^n\setminus\{0\}$ and a root of unity $\eta\in\mathbb{G}_m$ such that $\mathbf {X^u}-\eta | P.$

$3)\Rightarrow 1)$ We show the contraposition and assume that there is an $(n-1)$-dimensional algebraic subgroup $G\subset\mathbb{G}_m^n$ and a torsion point $\boldsymbol \zeta\in\mathbb{G}_m^n$ such that $\boldsymbol \zeta G\subset X$. By Theorem \ref{thm:bij} the rank of $\Lambda_G$ equals one. Thus there exists $\boldsymbol\lambda\in\mathbb{Z}^n\setminus\{0\}$ which is a basis of $\Lambda_G$. We can apply Lemma \ref{lem:artin} and find that $\boldsymbol{\zeta^\lambda}$ is a root of $P_\chi\in K[Y]$ and hence $Y-\boldsymbol{\zeta^\lambda}|P_\chi$ for all characters $\chi: G\to \mathbb G_m^n$ such that $L_\chi\ne \emptyset$. So let $Q_\chi\in K[Y^{\pm 1}]$ be such that $P_\chi=(Y-\boldsymbol{\zeta^\lambda})Q_\chi$. Then again by Lemma \ref{lem:artin} we find $$P=\sum_{\chi:L_\chi\ne \emptyset} \mathbf X^{\mathbf i_\chi}P_\chi(\mathbf X^{\boldsymbol\lambda})=(\mathbf X^{\boldsymbol\lambda}-\boldsymbol{\zeta^\lambda})\sum_{\chi:L_\chi\ne \emptyset} \mathbf X^{\mathbf i_\chi}Q_\chi(\mathbf X^{\boldsymbol\lambda}).$$ Therefore $\mathbf X^{\boldsymbol\lambda}-\boldsymbol{\zeta^\lambda}| P$, where $\boldsymbol\lambda\in\mathbb Z^n\setminus\{0\}$ and $\boldsymbol{\zeta^\lambda}\in\mathbb G_m$ is a torsion point.\end{proof}

\section{Bound for cosets}\label{sec:bdCoset}
Although we already esablished an asymptotic formula for $\#C_T$ if $C$ is a torsion coset, we prove an upper bound which is nice to work with, since it is very simple and there is no error term. 
\begin{definition}For an algebraic subgroup $G<\mathbb G_m^n$ we denote by $G^0$ the linear torus equal to the connected component of $G$ which contains the identity.\end{definition}

\begin{lemma}\label{lem:torsCoset}Let $\boldsymbol\zeta\in(\mathbb{G}_m^n)_{\mathrm{tors}}$ and let $G<\mathbb G_m^n$ be an algebraic subgroup of dimension $d$. Then we have $(\boldsymbol\zeta G)_T\le [G:G^0]T^{d+1}/\mathrm{ord}(\boldsymbol\zeta G)$.\end{lemma}
\begin{proof}Let us first assume that $G$ is a linear torus. In this case we can assume that $\boldsymbol{\zeta}G=\{\boldsymbol\eta\}\times \mathbb G_m^d$ for some $\boldsymbol\eta\in(\mathbb G_m^{n-d})_{\mathrm{tors}}$ of order $g:=\mathrm{ord}(\boldsymbol\zeta G)$ by Lemma \ref{lem:redTorsCoset}. Thus the order of $(\boldsymbol\eta,\boldsymbol\xi)$ is $\mathrm{lcm}(g,\mathrm{ord}(\boldsymbol\xi))$ for any $\boldsymbol\xi \in (\mathbb G_m^{d})_\mathrm{tors}$. Hence $(\boldsymbol\eta,\boldsymbol\xi)\in (\boldsymbol\zeta G)_T$ if and only if ord$(\boldsymbol\xi)\le \gcd(g,\mathrm{ord}(\boldsymbol\xi))T/g$. By Lemma \ref{lem:numberOfFixedOrder} the number of elements of order $N$ in $\mathbb{G}_m^d$ is $J_d(N)$ if $p=\mathrm{char}(K)=0$ or $(N,p)=1$ and zero otherwise. In particular $J_d(N)$ is always an upper bound. Hence we find $$\#(\boldsymbol\eta G)_T\le \sum_{1\le m\le \gcd(g,m)T/g} J_d(m)=\sum_{e|g}\sum_{m: \gcd(g,m)=e, m\le eT/g} J_d(m).$$
With the formula $J_d(n)=n^d\prod_{p|n}(1-p^{-d})$ one can see that $J_d(ab)\le a^dJ_d(b)$ for all $a,b\in\mathbb N$. And since $J_d=\mu\star I_d$, we have $\sum_{m|n}J_d(m)=n^d$ for all $n\in\mathbb N$ by the Möbius inversion formula. If $d|g$ and $m$ is as in the inner sum, then we can write $m=ke$ for some $k\le T/g$. Thus we find $$\#(\boldsymbol\eta G)_T\le \sum_{e|g}\sum_{1\le k\le T/g} J_d(ke)\le \sum_{e|g}J_d(e)\sum_{1\le k\le T/g} k^d\le g^d(T/g)^{d+1}\le T^{d+1}/g.$$

For the general case we decompose $\boldsymbol\zeta G$ into a disjoint union of $[G:G^0]$ torsion cosets $\boldsymbol\zeta_i G^0$ and observe that $\mathrm{ord}(\boldsymbol\zeta_i G^0)\ge \mathrm{ord}(\boldsymbol\zeta G)$ since $\boldsymbol\zeta_i G^0\subset \boldsymbol\zeta G$. Therefore $$\#(\boldsymbol\zeta G)_T=\sum_{i=1}^{[G:G^0]}(\boldsymbol\zeta_i G^0)_T\le \sum_{i=1}^{[G:G^0]}T^{d+1}/\mathrm{ord}(\boldsymbol\zeta_i G^0)\le[G:G^0]T^{d+1}/\mathrm{ord}(\boldsymbol\zeta G).\qedhere$$\end{proof}

\section{Hypersurfaces}\label{sec:hyp}

In this section we prove the main theorem for hypersurfaces. It remains to understand the points which are contained in a positive dimensional torsion coset. Since any such torsion coset is contained in a maximal one, we have to study them. Let $\boldsymbol\zeta G$ be a maximal coset. There are only finitely many possibilities for $G$. But it may happen that there are infinitely many torsion cosets $\boldsymbol\zeta G$ which are contained in the hypersurface. We construct a map from the torsion cosets of $G$ to torsion points of a variety of lower dimension and get a bound by induction.

\begin{lemma}\label{lem:torsCosets}Let $n\ge2$ and $X$ an admissible hypersurface. Let $G\subset\mathbb{G}_m^n$ be a linear torus of dimension $d$. Let $r=n-d$ and $L\in \mathrm{Mat}_{n\times r}(\mathbb Z)$ such that $L\mathbb Z^r=\Lambda_G$. Then there exist an admissible hypersurface $Y\subset \mathbb G_m^r$ and finitely many torsion cosets $C_1,\dots, C_m\subset\mathbb{G}_m^r$ of dimension at most $r-2$ with the property that for every $T\ge 1$ the map $\boldsymbol\zeta G\mapsto \boldsymbol\zeta^L$ is an injection from the torsion cosets $\boldsymbol\zeta G\subset X$ of order bounded by $T$ to $Y_T\cup \bigcup_{i=1}^m (C_i)_T$.\end{lemma}

\begin{proof}Without loss of generality we can assume that there exist torsion cosets of $G$ which are contained in $X$. Thus by admissibility we have $d\le n-2$ and hence $r\ge 2$. Let $P\in K[X_1^{\pm 1},\dots, X_n^{\pm 1}]$ such that $X=Z(P)$. Let $P_\chi\in K[Y_1^{\pm 1},\dots, Y_r^{\pm 1}]$ be defined as in Lemma \ref{lem:artin} for all $\chi:G\to \mathbb G_m^n$ such that $L_\chi\ne \emptyset$. For any such $\chi$ let $R_\chi$ be the product of all factors of $P_\chi$ of the form $\mathbf{ Y^u}-\zeta$ for some primitive $\mathbf u \in \mathbb Z^r\setminus \{0\}$ and $\zeta\in\mathbb G_m$ of finite order, which are irreducible by Lemma \ref{lem:irred}, counted with multiplicity. Let $Q_\chi\in K[Y_1^{\pm 1},\dots, Y_r^{\pm 1}]$ such that $P_\chi=Q_\chi R_\chi.$
Define $Q=\prod_{\chi:L_\chi\ne \emptyset} Q_\chi\in K[Y_1^{\pm 1},\dots,Y_r^{\pm 1}]$ and let $Y=Z(Q)\subset \mathbb G_m^r$. Since $r\ge 2$ we can apply Lemma \ref{lem:admP} and see that $Y$ is admissible by construction of $Q$.\\ We claim that the $R_\chi$ are coprime. Assume by contradiction that they have a common factor $F=\mathbf {Y^u}-\zeta$, where $\mathbf u\in\mathbb{Z}^r\setminus\{0\}$ and $\zeta$ is a root of unity and write $R_\chi=FR_\chi'$. Then we find by Lemma \ref{lem:artin} that $$P=\sum_{\chi:L_\chi\ne \emptyset}\mathbf X^{\mathbf i_\chi}P_\chi (\mathbf X^L)=F(\mathbf X^L)\sum_{\chi:L_\chi\ne \emptyset}\mathbf X^{\mathbf i_\chi}Q_\chi (\mathbf X^L)R'_\chi (\mathbf X^L).$$ 
Thus $\mathbf X^{L\mathbf u}-\zeta$ divides $P$. We now show that $L\mathbf u\ne 0$. Let $\chi$ be such that $L_\chi\neq \emptyset$. Then $P_\chi\neq 0$. Since $\mathbf u\neq 0, \zeta\neq 0$ and $F|P_\chi$ there exist $\mathbf{i\neq j}\in L_\chi$ and a natural number $k$ such that $k\mathbf {u=u_i-u_j}$ by Lemma \ref{lem:divide}. We have $kL\mathbf u = L(\mathbf{u_i-u_j)=i-j}\neq 0$ since $\chi_{\mathbf i}=\chi = \chi_{\mathbf j}$. Therefore $L\mathbf u \in\mathbb{Z}^n\setminus\{0\}$ and $\mathbf X^{L\mathbf u}-\zeta | P$ which contradicts the fact that $X$ is admissible by  Lemma \ref{lem:admP}. So we showed that the $R_\chi$ are coprime. We can therefore apply Lemma \ref{lem:R} and see that the common zeros in $\mathbb{G}_m^r$ is a finite union of torsion cosets $C_1,\dots, C_m$ of dimension at most $r-2$.\\
Let $T\ge 1$ and $\boldsymbol\zeta G$ be such that $\boldsymbol\zeta G\subset X$ and $\mathrm{ord}(\boldsymbol\zeta G)\le T.$ We claim that $\boldsymbol\zeta G\mapsto\boldsymbol\zeta^L$ is well defined and injective. This is a well defined map, because for $\mathbf x\in G$ we have $(\boldsymbol\zeta\mathbf x)^L=\boldsymbol\zeta^L$ since the columns of $L$ are in $\Lambda_G$. To prove injectivity assume that $\boldsymbol\zeta^L=\boldsymbol\xi^L$. Let $\mathbf x=\boldsymbol{\zeta\xi}^{-1}$ and $\boldsymbol\lambda\in\Lambda_G$. Then there exists $\mathbf u\in\mathbb{Z}^r$ such that $\boldsymbol\lambda=L\mathbf u$. Therefore $\mathbf x^{\boldsymbol\lambda}=\boldsymbol\zeta^{L\mathbf u}\boldsymbol\xi^{-L\mathbf u}=\mathbf 1^{\mathbf u}=1$ which shows that $\mathbf x\in G$ by Theorem \ref{thm:bij} and hence $\boldsymbol\zeta G=\boldsymbol\xi G$. Furthermore the order of the image is bounded by the order of the coset.

By Lemma \ref{lem:artin} we have $P_\chi(\boldsymbol\zeta^L)=0$. If there is $\chi$ such that $Q_\chi(\boldsymbol\zeta^L)=0$ the element $\boldsymbol\zeta^L$ is a root of $Q$ of order bounded by $T$ and hence in $Y_T$. If not, it is a common root of the $R_\chi$ and hence contained in the finite union of the torsion cosets of dimension at most $r-2$.\end{proof}

\begin{lemma}\label{lem:sum}Let $(a_n)\in\mathbb{R}_{\ge0}^{\mathbb N}$ be a sequence, $\alpha>1$ a real number and suppose that there is $c>0$ such that $\sum_{n=1}^{T} a_n \le c T^\alpha$ for all $T\in\mathbb N$. Then there is $C=C(c,\alpha)>0$ such that $\sum_{n=1}^{T} a_n/n \le C T^{\alpha-1}$ for all $T \in \mathbb N$.\end{lemma}
\begin{proof} Note that for all $n<T$ we have $\frac{1}{n}=\frac{1}{T}+\sum_{k=n}^{T-1}\left(\frac{1}{k}-\frac{1}{k+1}\right)=\frac{1}{T}+\sum_{k=n}^{T-1}\frac{1}{k(k+1)}$. Comparing the sum with the integral and using $\alpha-2>-1$ we can write \begin{align*}
\sum_{n=1}^T \frac{a_n}{n}&=\sum_{n=1}^{T-1}\sum_{k=n}^{T-1} \frac{a_n}{k(k+1)} +\frac{1}{T}\sum_{n=1}^T a_n
=\sum_{k=1}^{T-1} \frac{1}{k(k+1)}\sum_{n=1}^{k} a_n+\frac{1}{T}\sum_{n=1}^T a_n\\
&\leq c\sum_{k=1}^{T-1}{k}^{\alpha-2}+cT^{\alpha-1}\ll_\alpha cT^{\alpha-1}.\qedhere
\end{align*}\end{proof}

\begin{lemma}\label{lem:ind}Let $n\ge 3, X=Z(P)$ an admissible hypersurface in $\mathbb G_m^n$ and $1\le d\le n-2$. Assume that we know that $\#Y_T\ll_{Y} T^{r-1/r}$ for all admissible hypersurfaces $Y$ in $\mathbb G_m^r$, where $2\le r\le n-1$. Let $G<\mathbb G_m^n$ be a linear torus of dimension $d$. Then $$\# \left(\bigcup_{\substack{\boldsymbol\eta\in(\mathbb{G}_m^n)_{\text{tors}}\\\boldsymbol\eta G\subset X}} (\boldsymbol\eta G)_T\right)\ll_{X,G} T^{n-1/(n-d)} \text{ for all }T\ge 1.$$\end{lemma}
\begin{proof}Let us denote by $M$ the quantity on the left. Note that $\mathrm{ord}(\boldsymbol\eta G)>T$ implies that $(\boldsymbol\eta G)_T=\emptyset$. Thus we have only to consider cosets of order bounded by $T$. Let $a_N$ be the number of torsion cosets $\boldsymbol\eta G$ such that $\boldsymbol\eta G\subset X$ and $\mathrm{ord}(\boldsymbol\eta G)= N.$ Then we can use Lemma \ref{lem:torsCoset} to find the bound $M\ll_G T^{d+1}\sum_{N=1}^T a_N/N$. Let $r=n-d\le n-1$ and $Y\subset\mathbb{G}_m^r$ and $C_1,\dots, C_m \subset\mathbb{G}_m^r$ be the admissible hypersurface and the torsion cosets from Lemma \ref{lem:torsCosets}. By hypothesis there exists a constant $c=c(Y)$ such that $\#Y_T\le cT^{r-1/r}$ and by Lemma \ref{lem:torsCoset}  we have $\#(C_i)_T\le T^{r-1}$ for all $T\ge 1$. So by Lemma \ref{lem:torsCosets} for all $N\ge 1$ we have $$\sum_{N=1}^T a_N\le \#(Y)_T+\sum_{i=1}^m \#(C_i)_T\le (c+m)T^{r-1/r}.$$
Since $r\ge 2$ we have $r-1/r\ge 3/2>1$ and therefore $\sum_{N=1}^T a_N/N\ll_{Y} T^{r-1/r-1}$ by Lemma \ref{lem:sum}.
So we get $M\ll_{X,G} T^{d+1+r-1-1/r}=T^{n-1/{(n-d)}}.$\end{proof}

\begin{lemma}Let $n\in\mathbb{N}$ and $X=Z(P)$ a hypersurface in $\mathbb G_m^n$. Then $\#X_T\ll_X T^{n}.$ If $n=1$ we even have $\#X_T\ll_X 1$.\end{lemma}
\begin{proof}The proof is by induction on $n$. If $n=1$ the set $Z(P)$ is finite and hence the lemma is true in this case. So assume that $n\ge 2$ and the lemma is true in dimension smaller than $n$. Write $P=\sum_{i=-D}^D R_i(X_1,\dots, X_{n-1})X_n^i$. Let $R\in\{R_{-D}, R_D\}\subset K[X_1^{\pm 1},\dots,X_{n-1}^{\pm 1}]$ be a nonzero Laurent polynomial which exists since $P\ne 0$. To estimate $\#Z(P)_T$ we divide $X_T$ into two sets: $Y_1=\{(x_1,\dots,x_n)\in X_T: R(x_1,\dots, x_{n-1})\ne 0\}$ and $Y_2=X_T\setminus Y_1$. First we bound $\#Y_1$. Fix $\mathbf x=(x_1,\dots, x_{n-1})\in \mathbb G_m^{n-1}$ and assume that $R(\mathbf x)\ne 0$. Then there are at most $2D$ solutions $x\in\mathbb G_m$ satisfying $P(\mathbf x, x)=0$. We can bound $\#Y_1\le 2D(\mathbb G_m^{n-1})_T\le 2DT^n$ by Lemma \ref{lem:torsCoset}. If $R\in K[X_1^{\pm 1},\dots, X_n^{\pm1}]^*$, then we have $X_T=Y_1$ and the claimed bound holds.\\
Otherwise we also have to bound $\#Y_2$ and $V=Z(R)$ is a hypersurface. Let $G=\{1\}\times \dots\times\{ 1\}\times \mathbb G_m$. Then $$Y_2 \subset \bigcup_{\mathbf x\in V_T} ((\mathbf x,1)G)_T.$$ The order of a coset $(\mathbf x,1)G$ equals $\mathrm{ord}(\mathbf x)$. Let $c$ be a constant such that $\#V_T\le cT^{n-1}$ for all $T$ which exists by induction. Let $a_N$ be the number of points of order $N$ in $V$. Thus by Lemma \ref{lem:torsCoset} we have to bound the sum $$\#Y_2\ll  \sum_{\mathbf x\in V_T}T^2/\mathrm{ord}(\mathbf x)=T^2 \sum_{N=1}^T a_N/N.$$ By induction we know that $\sum_{N=1}^T a_N=\#V_T \ll_V 1$ if $n=2$ and $\sum_{N=1}^T a_N=\#V_T\ll_V T^{n-1}$ if $n>2$. So if $n=2$ we find $\#Y_2\ll T^2 \sum_{N=1}^T a_N/N \ll_R T^2$. If $n>2$ we have $n-1>1$, so we can apply Lemma \ref{lem:sum} and find $\# Y_2\ll T^2 \sum_{N=1}^T a_N/N \ll_{n,R} T^{2+(n-1)-1}=T^n$. Thus we have $\#Y_2\ll_{n,P} T^n$ and hence $\#X_T=\#Y_1 +\#Y_2 \ll_X T^n$.\end{proof}

\begin{theorem}\label{thm:pol}Let $n\ge 2$, $X$ be an admissible hypersurface in $\mathbb G_m^n$. Then $\#X_T\ll_{X} T^{n-1/n}.$\end{theorem}
\begin{proof}
Since any torsion coset in $X$ is contained in a maximal one, we find $$X\cap(\mathbb G_m^n)_{\mathrm{tors}}\subset X^*\cup \bigcup_{\substack{\boldsymbol\zeta G\subset X\text{ maximal torsion coset},\\ \dim(G)\ge 1}} \boldsymbol \zeta G.$$ Let $\mathcal G=\{G<\mathbb G_m^n: G \text{ linear torus, there is } \boldsymbol\zeta\in (\mathbb G_m^n)_{\mathrm{tors}}: \boldsymbol\zeta G \subset X\text{ is maximal and }\dim(G)\ge 1\}$ which is finite by Corollary \ref{cor:fin}. Then we have $$\#X_T\le \#X^*_T+\sum_{G\in\mathcal G}\# \left(\bigcup_{\substack{\boldsymbol\zeta\in(\mathbb{G}_m^n)_{\text{tors}}\\\boldsymbol\zeta G\subset X}} (\boldsymbol\zeta G)_T\right)$$ and since $\mathcal G$ is finite it is enough to show the bound for each summand separately. The proof is by induction on $n$.
 If $n=2$ we have $X_T= X^*_T$ and the theorem follows from Lemma \ref{lem:D0}. Let $n\ge 3$ and $G\in \mathcal G$ be of dimension $d$. We have $1\le d\le n-2$ by admissibility. Thus the hypothesis of Lemma \ref{lem:ind} is satisfied by induction. We can apply it and find $\# \left(\bigcup_{\substack{\boldsymbol\eta\in(\mathbb{G}_m^n)_{\text{tors}}\\\boldsymbol\eta G\subset X}} (\boldsymbol\eta G)_T\right)\ll_{X,G} T^{n-1/(n-d)}$. Thus we have $\#X_T\ll_X T^{n-1/(n-d)}$.\end{proof}

\section{Generalisation}\label{sec:gen}

We deduce the general theorem from the case where the subvariety is a hypersurface using algebraic geometry.

\begin{lemma}\label{lem:XYZ}Let $n\ge 2$, $X\subset \mathbb{G}_m^n$ be an algebraic set of dimension $d$ and $\pi:X\to\mathbb{G}_m^{n-1}$ the projection to the first $n-1$ coordinates. Let $Z=\{\mathbf p\in\mathbb G_m^{n-1}:\dim (\pi^{-1}(\mathbf p))\ge 1\}$. Then there exists $D>0$ such that for all $\mathbf{p}\in\pi(X)$ we have $\#\pi^{-1}(\mathbf p)\le D$ or $\mathbf p\in Z$. The set $Z$ is an algebraic set of dimension at most $d-1$. The Zariski closure $Y=\overline{\pi(X)}$ of the image is of dimension at most $d$.\end{lemma}
\begin{proof}Let $\mathbf p\in \pi(X)$. The number of irreducible components in the fibres is uniformly bounded and hence there exists $D>0$ such that all finite fibres, contain at most $D$ points. Note that the fibres of positive dimension are of the form $\mathbf q\times \mathbb G_m$ and therefore $V=\{\mathbf x\in X: \dim_{\mathbf x} \pi^{-1}(\pi(\mathbf x))\ge 1\}=A\times \mathbb G_m$ for some $A\subset \mathbb G_m^{n-1}$. By Theorem 14.112 in \cite{GW} we have that $V$ and hence also $A$ is closed. Thus $Z=\pi(V)=A$ is closed. Since $Z\times \mathbb G_m\subset X$ we have $\dim(Z)\le d-1$. That the dimension of the closure of a morphism is at most the dimension of the domain is well known.\end{proof}

\begin{proof}[Proof of Proposition \ref{thm:gen}] We prove $\#X_T\le \sum_{1\le m\le T}\#\{\mathbf x\in X: \mathbf x^m=\mathbf 1\}\ll_X T^{d+1}$. The first inequality follows immediately. It suffices to show the second one.\\Let $X_1,\dots, X_r$ be the irreducible components of $X$. Then $\max_{1\le i\le r}\dim(X_i)= d$ and $X_T=\bigcup_{i=1}^r (X_i)_T$ for all $T\ge 1$. Therefore we can assume that $X$ is irreducible.\\
If $d=0$ then $X$ is finite and the claim is immediate. If $n=1$ and $d=1$, then $X=\mathbb G_m$ and the sum is bounded by $\sum_{m=1}^T m\le T^2$. So we may assume $d\ge 1$ and $n\ge 2$\\
We continue by induction on $n+d$, the case $n+d=1$ is already done.\\
Let $\pi:\mathbb G_m^n\to \mathbb G_m^{n-1}$ denote the projection to the first $n-1$ coordinates. The set $$Z=\{\mathbf p\in\mathbb G_m^{n-1}:\dim (\pi|_X^{-1}(\mathbf p))\ge 1\}$$ is Zariski closed and of dimension at most $d-1$ by Lemma \ref{lem:XYZ}.
Let $T\in\mathbb N$ and observe that the sum in question is at most $a_T + b_T$, where $$a_T=\sum_{m=1}^T\#\{(\mathbf p,t)\in X: \mathbf p\in Z,\mathbf p^m=\mathbf 1, t^m=1\}\text{ and }b_T=\sum_{m=1}^T\#\{(\mathbf p,t)\in X: \mathbf p\not\in Z,\mathbf p^m=\mathbf 1, t^m=1\}.$$
Then $$a_T\le \sum_{m=1}^T\# \{\mathbf p\in Z: \mathbf p^m=\mathbf 1\} m\le T\sum_{m=1}^T\# \{\mathbf p\in Z: \mathbf p^m=\mathbf 1\}.$$ By induction applied to $Z\subset \mathbb G_m^{n-1}$ we have $a_T\ll_Z T^{\dim(Z) +1}\le T^{d+1}$.\\
If $\mathbf p\in\mathbb G_m^{n-1}\setminus Z$ there are at most $D$ possibilities for $t\in\mathbb G_m$ such that $(\mathbf p,t)\in X$ by Lemma \ref{lem:XYZ}. We apply the current lemma by induction to the Zariski closure $\overline{\pi(X)}\subset \mathbb G_m^{n-1}$ which has dimension at most $d$. So $b_T\le D\sum_{m=1}^T \#\{\mathbf p\in\overline{\pi(X)}: \mathbf p^m=\mathbf 1\}\ll_X T^{d+1}$.
\end{proof}

\begin{definition}Let $X$ be a topological space. Then we denote by $\mathrm{Irr}(X)$ the set of irreducible components of $X$.\end{definition}

\begin{remark}\label{rk:hypSurf}If an algebraic subset $X\subset \mathbb G_m^n$ is irreducible and of dimension $n-1$, it is a hypersurface. This can be deduced from the same fact for subvarieties of $\mathbb A^n$, which is stated in Theorem 1.21 in \cite{Shaf}.\end{remark}

\begin{theorem}\label{thm:sineStab}Let $K$ be an algebraically closed field and $X\subset \mathbb G_m^n$ irreducible and admissible of dimension $d$. Then we have $\#X_T\ll T^{d+1-1/(d+1)}$ for all $T\ge 1$.\end{theorem}

\begin{proof}As the conclusion of the theorem is invariant under permuting coordinates
we may freely do so. As $\mathbb G_m^n$ itself is not admissible we may assume $d\le n-1$. If $d=0$, the set $X$ is finite and hence we have $X_T\le \#X \ll_X 1$. So we can assume that $d\ge 1$.\\
Next we fix a good projection. To this end we consider the coordinate functions
$x_1, \dots , x_n$ restricted to $X$. They are denoted $x_1|_X , \dots , x_n|_X$ and members of the
function field of $X$. By dimension theory, the function field $K(x_1|_X , \dots , x_n|_X )$ has
transcendence degree $d$ over $K$. After permuting coordinates we may assume that
$x_1|_X ,\dots, x_d|_X$ are algebraically independent.

We claim that there exists $k\in\{d+1,\dots, n\}$ such that $x_1|_X ,\dots, x_d|_X, x_k|_X$ are multiplicatively independent. Let us assume the contrary and deduce a contradiction. For each such $k$ there exist $a_{k,1},\dots, a_{k,d}, a_{k,k}\in\mathbb{Z}$ not all zero such that \begin{equation}\label{eq:mult}x_1|_X^{a_{k,1}}\cdots x_d|_X^{a_{k,d}}x_k|_X^{a_{k,k}}=1\end{equation} identically on $X$. By algebraic independence we must have $a_{k,k}\ne 0$. So the $n-d$ vectors $(a_{k,1},\dots,a_{k,d},0,\dots,0,a_{k,k}, 0,\dots, 0)\in\mathbb Z^n$ for $k\in\{d+1,\dots, n\}$ are linearly independent. They generate a subgroup $\Lambda<\mathbb Z^n$ of rank $n-d$. The algebraic subgroup $H=H_\Lambda<\mathbb G_m^n$ is of dimension $d$. By \eqref{eq:mult} we have $X\subset H$. But since $X$ is irreducible and of the same dimension as $H$, it is an irreducible component of $H$ and hence a torsion coset. This contradicts the admissibility of $X$.

After permuting coordinates we may assume $k = d + 1$.
Let $\phi : \mathbb G_m^n\to \mathbb G_m^{d+1}$ denote the projection onto the first $d + 1$ coordinates. Then
the Zariski closure $Y=\overline{\phi(X)}$ of $\phi(X)$ in $\mathbb G_m^{d+1}$ is irreducible and $\dim(Y)\le d$. For all $(\mathbf{x,x}')\in X$ we have $x_i|_X(\mathbf{x,x}') = x_i|_Y(\mathbf x)$. Thus $x_1|_Y,\dots, x_d|_Y$ must be algebraically inedependent elements of the function field of $Y$ and hence we have $\dim(Y)= d$.
Next we show that $Y$ is admissible. Assume that it is not. Then there exists a torsion coset $\boldsymbol\zeta G\subset Y$ of dimension $d$.  Since both $\boldsymbol\zeta G$ and $Y$ are irreducible we have $Y=\boldsymbol\zeta G$. Then there exists $\boldsymbol\lambda\in\mathbb Z^{d+1}\setminus\{0\}$ such that $\mathbf y^{\boldsymbol\lambda}=\boldsymbol{\zeta^\lambda}$ for all $\mathbf y\in Y$. Let $N$ be the order of $\boldsymbol\zeta$. Then since $(x_1,\dots,x_{d+1})\in Y$ for all $(x_1,\dots, x_n)\in X$ we find $x_1|_X^{N\lambda_1}\cdots x_{d+1}|_X^{N\lambda_{d+1}}=1$ identically on $X$, a contradiction. Hence $Y$ is admissible. By Remark \ref{rk:hypSurf} we have that $Y$ is a hypersurface. 

Let $$Z=\{\mathbf x\in X: \dim_{\mathbf x}\left(\phi|_X^{-1}(\phi(\mathbf x))\right)\ge 1\}.$$
It is an algebraic subset of $\mathbb G_m^n$ by Theorem 14.112 in \cite{GW}. The set $V=\{\mathbf y\in \mathbb G_m^{d+1}: \dim\left(\phi|_X^{-1}(\mathbf y)\right)\ge 1\}$ is algebraic and so is $Z=\phi|_X^{-1}(V)$. By Corollary 14.116 in \cite{BG} there is a nonempty open $U$ of $Y$ such that $\dim(\phi|_X^{-1}(\mathbf u))=0$ for all $\mathbf u\in U$. Since $\phi(X)$ is dense in $Y$, its intersection with $U$ is nonempty and therefore $Z$ is a proper subset of $X$. In particular we find $\dim(Z)\le d-1$. 

Let $T\ge 1$. Then $X_T\subset Z_T\cup B_T$ where \begin{align*}B_T&=\{\mathbf x\in X: \dim_{\mathbf x}(\phi|_X^{-1}(\phi(\mathbf x)))=0, \mathrm{ord}(\phi(\mathbf x))\le T\}.\end{align*} Applying Proposition \ref{thm:gen} we find $\#Z_T\ll_Z T^{\dim(Z)+1}\le T^d$ for all $T\ge 1$. Let $C$ be a uniform bound for the number of irreducible components in a fibre. For any $\mathbf x\in B_T$ we have $\mathbf y=\phi(\mathbf x)\in Y_T$ and $\{x\}$ is an irreducible component of $\phi|_X^{-1}(\mathbf y)$. Thus we find $$\#B_T\le\sum_{\substack{\mathbf y\in Y_T \\ \dim(\phi|_X^{-1}(\mathbf y))=0}} \#\phi|_X^{-1}(\mathbf y)\le \sum_{\mathbf y\in Y_T } \#\mathrm{Irr}(\phi|_X^{-1}(\mathbf y))\le C\#Y_T.$$ 
Since $Y$ is an admissible hypersurface in $\mathbb G_m^{d+1}$ we can apply Theorem \ref{thm:pol} to bound $\#Y_T\ll_X T^{d+1-1/(d+1)}$ for all $T\ge 1$. The theorem follows from $\#X_T\le \#Z_T+\#B_T$.\end{proof}

\begin{lemma}\label{lem:dimPreIm}Let $\phi: G\to G'$ be a surjective morphism of algebraic groups. Let $H'<G'$ be a subgroup and $g'\in G'$.
Then $\dim(\phi^{-1}(H'))=\dim(\ker \phi)+\dim (H')$.\end{lemma}
\begin{proof}Let $H=\phi^{-1}(H')$ and $\psi:H\to H'$ which maps $g$ to $\phi(g)$. Since $\phi$ is surjective, $\psi$ is surjective too. We have $\dim H=\dim(\ker(\psi))+\dim H'$ and since $\ker(\phi)=\ker(\psi)$ the lemma follows.
\end{proof}

\begin{proof}[Proof of Theorem \ref{thm:main}] If $d=0$, $X$ and $G=\mathrm{Stab}(X)$ are both finite. The inequality is satisfied, since the exponent equals $0$. Thus we can assume that $d\ge 1$. If $\delta=d$ let $\mathbf x\in X$ and $H$ be the connected component of the identity of $G$. The coset $\mathbf xH\subset X$ is irreducible and of dimension $d$. Thus we find $X=\mathbf xH$. Since $X$ is admissible $\mathbf x$ is not a torsion point and hence $\#X_T=0$. Thus we can assume that $\delta<d\le n-1$. 

Let $L\in\mathrm{Mat}_{n\times (n-\delta)}(\mathbb Z)$ be a basis of $\Lambda_G$. Let $\phi:\mathbb G_m^n\to \mathbb G_m^{n-\delta}$ given by $\mathbf x\mapsto \mathbf x^L$. Note that $\phi$ is surjective since $\dim (\mathrm{Im}(\phi))=n-\dim(G)=n-\delta$. Let $Y\subset \mathbb G_m^{n-\delta}$ be the Zariski closure of $\phi(X)$. Then $Y$ is irreducible and hence either admissible or a torsion coset. Since the fibres of $\phi$ are cosets of $G$ we have $\dim(Y)=d-\delta$. Assume by contradiction that $Y$ is a torsion coset. Then there exists a torsion point $\boldsymbol\zeta\in \mathbb G_m^{n-\delta}$ of order $N$ and a linear torus $H<\mathbb G_m^{n-\delta}$ such that $Y=\boldsymbol\zeta H$. Then $H'=\bigcup_{\boldsymbol\zeta^N=1}\boldsymbol\zeta H$ is an algebraic subgroup of $\mathbb G_m^{n-\delta}$ of dimension $d-\delta$. By Lemma \ref{lem:dimPreIm} we have $\dim(\phi^{-1}(H'))=\delta+(d-\delta)=d$. Since $\phi(X)\subset Y=\boldsymbol\zeta H \subset H'$ and $d=\dim X$ we see that $X$ is an irreducible component of $\phi^{-1}(H')$ which are torsion cosets by Lemma \ref{lem:HLambda}. Since $d\ge 1$ this contradicts admissibility. Thus $Y$ is admissible.

Let $a_k$ be the number of cosets of $G$ of order $k$ contained in $X$. We have $$X_T=\bigcup_{\substack{\boldsymbol\zeta G \subset X\text{ torsion coset}\\\mathrm{ord}(\boldsymbol\zeta G)\le T}}(\boldsymbol\zeta G)_T \text{ and hence }\#X_T\ll_G T^{\delta+1}\sum_{k=1}^T a_k/k\text{ for all } T\ge 1$$ using Lemma \ref{lem:torsCoset}. If $\mathbf xG\subset X$ is a coset of order $k$, it is the preimage of $\mathbf x^L$ which is of order at most $k$. Therefore $\phi$ induces an injection $\mathbf xG\mapsto \phi(\mathbf x)$ from the cosets of $G$ which are contained in $X$ and of order bounded by $T$ into $Y_T$. Thus we find $\sum_{k=1}^T a_k\le \#Y_T \ll_X T^{d-\delta+1-1/(d-\delta+1)}$ for all $T\ge 1$ by Theorem \ref{thm:sineStab}. Since $d-\delta\ge 1$ the exponent is strictly larger than $1$ and using Lemma \ref{lem:sum} we find $$\#X_T\ll_G T^{\delta+1}\sum_{k=1}^T a_k/k\ll_X T^{(\delta+1)+(d-\delta+1-1/(d-\delta+1))-1}=T^{d+1-1/(d-\delta+1)}.\qedhere$$
\end{proof}

\begin{corollary}\label{cor:main}Let $K$ be an algebraically closed field and $X\subset\mathbb G_m^n$ admissible of dimension $d$. Let $\delta=\min_{Y\in \mathrm{Irr}(X)} \dim(\mathrm{Stab}(Y))$. Then $\#X_T\ll T^{1+d-1/(d-\delta+1)}$.\end{corollary}
\begin{proof}Note that $\#X_T\le \sum_{Y\in\mathrm{Irr}(X)} \#Y_T$ and apply Theorem \ref{thm:main} to each summand.\end{proof}
\begin{remark}The number $\delta$ in Corollary \ref{cor:main} is also equal to $$\min_{\mathbf x\in X}\max\{\dim(G): G\text{ linear torus, } \mathbf xG\subset X\}.$$\end{remark}

\section{Characteristic 0}\label{sec:char0}

In characteristic zero we can say much more than in general: The Manin Mumford Conjecture allows us to reduce the general case to the case of torsion cosets. Thus we can give an asymptotic formula for any subvariety. In this section $K$ denotes a field of characteristic $0$.

The following theorem is sometimes called the Manin-Mumford- conjecture. In the current setting it was proved by Laurent.

\begin{theorem}\label{thm:MM}Let $X\subset \mathbb{G}_m^n$ be an algebraic set. Then there are finitely many torsion cosets $C_1,\dots, C_m$ such that $X\cap(\mathbb{G}_m^n)_{\mathrm{tors}} = \bigcup_{i=1}^m (C_i\cap (\mathbb{G}_m^n)_{\mathrm{tors}})$.\end{theorem}
\begin{proof}This is Théorème 2 in \cite{ML}.\end{proof}

\begin{proof}[Proof of Corollary \ref{cor:char0}]By Theorem \ref{thm:MM} there exist finitely many torsion cosets $C_1,\dots, C_m$ in $\mathbb{G}_m^n$ such that $X\cap(\mathbb{G}_m^n)_{\mathrm{tors}} = \bigcup_{i=1}^m (C_i\cap (\mathbb{G}_m^n)_{\mathrm{tors}})$. Intersecting both sides with $(\mathbb G_m^n)_T$ we get $X_T=\bigcup_{i=1}^m (C_i)_T$. If $d=0$, there are at most finitely many torsion points in $X$ and thus we have $\#X_T\ll_X 1$ in this case. If $d\ge 1$ the corollary follows directly from Theorem \ref{thm:countCoset} since the intersection of the torsion cosets is a finite union of torsion cosets of strictly lower dimension.\end{proof}

\section{Lower bounds in $\overline{\mathbb F_p}$}\label{sec:charp}

We use the Lang-Weil estimates to establish lower bounds if $K=\overline{\mathbb F}_p$. 

\begin{theorem}\label{thm:LW}Let $X\subset \mathbb{P}^n$ be a geometrically irreducible variety defined over a finite field $K$ of positive characteristic with $q$ elements. Assume that $X$ is of dimension $r$. Then there exists $C=C(X)$ such that $|\#X(K)-q^r|\le Cq^{r-(1/2)}$.\end{theorem}
\begin{proof}This is Theorem 1 in \cite{LW}.\end{proof}

\begin{proof}[Proof of Theorem \ref{thm:charp}]Considering an irreducible component of $X$ of dimension $d$ we can assume that $X$ is irreducible. Let $Z$ be the Zariski closure of $Y=\{(1:x_1:\dots:x_n)\in\mathbb P^n: (x_1,\dots, x_n)\in X\}$ in $\mathbb P^n$. Then $Z$ is irreducible and of dimension $d$. Since $Z$ is irreducible, the proper and closed subset $B=Z\cap\{(x_0:\dots:x_n)\in\mathbb P^n: x_0\cdots x_n=0\}$ is of dimension at most $d-1$ and we have $Z=Y\cup B$. 

Let $k\in\mathbb{N}$ be such that $Z$ is defined over $\mathbb F_{p^k}$. Let $l\in\mathbb{N}$ be a multiple of $k$. Then $\mathbb F_{p^k}$ is a subfield of $\mathbb F_{p^l}$ and hence $Z$ and $B$ are defined over $\mathbb F_{p^l}$ too. By Theorem \ref{thm:LW} there exist constants $C,D>0$ independent of $l$ such that $|\#Z(\mathbb F_{p^l})-p^{ld}|\le Cp^{l(d-1/2)}$ and $\#B(\mathbb F_{p^l})\le Dp^{l\dim(B)} \le Dp^{l(d-1)}$. Let $T\gg 1$ be large enough and $l$ a multiple of $k$ such that $p^{l}-1\le T \le p^{l+k}-1$. Note that for all $(x_0:\dots:x_n)\in Z(\mathbb F_{p^l})\setminus B(\mathbb F_{p^l})$ we have $(x_1/x_0,\dots, x_n/x_0)\in X_{p^l-1}$ since for all nonzero $x\in\mathbb F_{p^l}$ we find $x^{p^l-1}=1$. Therefore we get \begin{align*}\#X_T&\ge \#X_{p^l-1}\ge \#Z(\mathbb F_{p^l})-\#B(\mathbb F_{p^l})\\
&\ge p^{ld}-Cp^{l(d-1/2)}-D p^{l(d-1)}=p^{ld}(1-Cp^{-l/2}-Dp^{-l}).\end{align*}Thus if $T$ and hence $l$ is large enough, we have $\#X_T\ge \frac{p^{ld}}{2}\ge \frac{1}{2p^{kd}}(p^{k+l})^d\ge \frac{1}{2p^{kd}}T^d$.\end{proof}

\appendix
\section{Algebraic subgroups and subgroups of $\mathbb Z^n$}\label{sec:algSubg}

We verify that many results of chapter 3 in \cite{BG} remain valid over arbitrary algebraically closed fields. These facts may be well-known, but we are not aware of a reference. In this section $K$ denotes an algebraically closed field.

\begin{definition}For an algebraic subgroup $G$ we define $$\Lambda_G=\{\mathbf v\in\mathbb Z^n: \mathbf{x^v}=1 \text{ for all }\mathbf x\in G\}$$ and for a subgroup $\Lambda<\mathbb Z^n$ we set $$H_\Lambda = \{\mathbf x\in \mathbb G_m^n: \mathbf{x^v}=1 \text{ for all } \mathbf v\in\Lambda\}.$$\end{definition}
\begin{definition}Let $\Lambda<\mathbb Z^n$ be a subgroup. The saturation $\tilde \Lambda$ of $\Lambda$ is defined as $$\tilde\Lambda=\{\mathbf v\in \mathbb Z^n: \text{ there exists }k\in\mathbb{N}: k\mathbf v\in \Lambda\}.$$ We call $\Lambda$ primitive if $[\tilde\Lambda:\Lambda]=1$.\end{definition}

\begin{definition}Let $p$ be a prime or $0$ and $\Lambda < \mathbb Z^n$ a subgroup. We say that $\Lambda$ is $p$-full if $p=0$ or if $p$ is prime and for all $\mathbf v\in\mathbb Z^n$ we have that $p\mathbf v\in\Lambda$ implies that $\mathbf v\in\Lambda$.\end{definition}

\begin{lemma}\label{lem:HLambda}Let $p=\mathrm{char}(K)$ and $\Lambda$ be a $p$-full subgroup of $\mathbb Z^n$ of rank $n-r$. Then $H_\Lambda$ is an algebraic subgroup of $\mathbb{G}_m^n$ of dimension $r$, which is the union of $[\tilde \Lambda:\Lambda]$ torsion cosets of dimension $r$. We also have $\Lambda_{H_\Lambda}=\Lambda$.\end{lemma}
\begin{proof}By the theorem of elementary divisors (Theorem III.7.8 in \cite{Lang}) there exists a basis $\mathbf b_1,\dots,\mathbf b_n$ of $\mathbb Z^n$ and $\lambda_1,\dots, \lambda_{n-r}\in \mathbb{Z}\setminus \{0\}$ such that $\lambda_1\mathbf b_1,\dots, \lambda_{n-r}\mathbf b_{n-r}$ is a basis of $\Lambda$. Let $B=(\mathbf b_1,\dots,\mathbf b_n)\in\mathrm{GL}_n(\mathbb Z)$. Then the monoidal transform $\mathbf x\mapsto \mathbf x^B$ gives an isomorphism between $H_\Lambda$ and $F\times \mathbb G_m^r$ where $F=\{\mathbf x\in\mathbb G_m^{n-r}: x_1^{\lambda_1}=1,\dots, x_{n-r}^{\lambda_{n-r}}=1 \}$ since $\mathbf x^{\lambda_i\mathbf b_i}=1$ if and only if $\mathbf x^{\lambda_i B\mathbf e_i}=1$ and thus if and only if $\mathbf {y=x}^B$ satisfies $\mathbf y_i^{\lambda_i}=1$. Then we see that $H_\Lambda$ is a union of $\#F$ linear tori of dimension $r$.

We claim that $\mathrm{span}\{\mathbf b_1,\dots,\mathbf b_{n-r}\}=\tilde\Lambda$. Since $\lambda_i\mathbf b_i\in \Lambda$ the inclusion $\subset$ holds. For the other let $\mathbf v=\alpha_1\mathbf b_1+\dots+\alpha_n\mathbf b_n\in \tilde\Lambda$. There exists $m\in\mathbb{N}$ such that $m\mathbf v\in\Lambda.$ Hence we can write $m\mathbf v= \beta_1\lambda_1\mathbf b_1+\dots+\beta_{n-r}\lambda_{n-r}\mathbf b_{n-r}$ and the comparison of coefficients yields $\alpha_{n-r+1}=\dots=\alpha_n=0$. Hence we have $\mathbf v\in \mathrm{span}\{\mathbf b_1,\dots, \mathbf b_{n-r}\}$ and the claim is true.

Consider the surjective group homomorphism $\mathbb{Z}^{n-r}\to \tilde\Lambda$ given by $(\alpha_1,\dots,\alpha_{n-r})\mapsto \alpha_1\mathbf b_1+\dots +\alpha_{n-r}\mathbf b_{n-r}$ and compose it with the reduction $\tilde\Lambda\to\tilde\Lambda/\Lambda$. Let $D=\mathrm{diag}_{n-r,n-r}(\boldsymbol\lambda)$. Then the kernel of the composition is $\lambda_1\mathbb{Z}\times \dots \times \lambda_{n-r}\mathbb{Z}=D\mathbb{Z}^{n-r}$. Thus we find $[\tilde\Lambda:\Lambda]=[\mathbb Z^{n-r}:D\mathbb{Z}^{n-r}]=\prod_{i=1}^{n-r}|\lambda_i|$.
Since $\Lambda$ is $p$-full the $\lambda_i$ are coprime to $p$, if $p>0$ and therefore we have $\#\mu_{|\lambda_i|}=|\lambda_i|$ in all cases. Thus we find $\#F=\prod_{i=1}^{n-r}\#\mu_{|\lambda_i|}=\prod_{i=1}^{n-r}|\lambda_i|=[\tilde\Lambda:\Lambda]$.

Let $\Lambda'=\lambda_1\mathbb Z\times \dots \times \lambda_{n-r}\mathbb Z\times \{0\}\times \dots\times \{0\}<\mathbb Z^n$. In a first step we reduce the last statement to $\Lambda_{H_{\Lambda'}}=\Lambda'$. We have seen above that $\mathbf x\mapsto\mathbf x^B$ is an isomorphism between $H_\Lambda$ and $H_{\Lambda'}$. Therefore we have $$\Lambda_{H_\Lambda}=\{\mathbf v\in\mathbb Z^n: \mathbf{x^v}=1 \text{ for all }\mathbf x\in H_\Lambda\}=\{BB^{-1}\mathbf v\in\mathbb Z^n: \mathbf x^{B^{-1}\mathbf v}=1 \text{ for all }\mathbf x\in H_{\Lambda'}\}=B\Lambda_{H_{\Lambda'}}.$$ So if we assume that $\Lambda_{H_{\Lambda'}}=\Lambda'$ we find $\Lambda_{H_\Lambda}=B\Lambda'=\Lambda$.\\
In the second step we prove $\Lambda_{H_{\Lambda'}}=\Lambda'$ and let $\mathbf v\in\Lambda'$. Then we have $\mathbf{ x^v}=1$ for all $\mathbf x\in H_{\Lambda'}$ and hence $\mathbf v\in \Lambda_{H_{\Lambda'}}$.

For the other inclusion let $\mathbf v\in\Lambda_{H_{\Lambda'}}$. Let $i\in\{1,\dots, n-r\}$. Since $\lambda_i$ is coprime to $p$ if $p>0$ there always exists a root of unity $\xi_{\lambda_i}$ of order $|\lambda_i|$. We have $(1,\dots,1,\xi_{\lambda_i},1,\dots, 1)\in H_{\Lambda'}$ where $\xi_{\lambda_i}$ is the $i$-th coordinate. Hence $\xi_{\lambda_i}^{v_i}=1$ which shows that $v_i\in \lambda_i\mathbb Z$. Now let $i\in\{n-r+1,\dots, n\}$. Then for all $x\in\mathbb G_m$ we have $(1,\dots,1,x,1,\dots,1)\in H_{\Lambda'}$ where $x$ is the $i$-th coordinate and hence $x^{v_i}=1$. This implies that $v_i=0$, and thus $\mathbf v\in \Lambda'$ and the other inclusion holds too. 
\end{proof}

\begin{definition}A torus coset is a coset of a linear torus.\end{definition}

\begin{lemma}\label{lem:3214}Let $X\subset \mathbb G_m^n$, be an algebraic set defined by Laurent polynomials $P_i:=\sum_{\boldsymbol\lambda} a_{i,\boldsymbol\lambda}\mathbf{x}^{\boldsymbol\lambda}=0$ and let $L_i$ be the set of exponents appearing in the monomials in $P_i, i=1,\dots m$. Let $\mathbf zG$ be a maximal coset in $X$. Then $G=H_{\Lambda}$, where $\Lambda$ is a subgroup generated by vectors of type $\boldsymbol\lambda'_i-\boldsymbol\lambda_i$ with $\boldsymbol\lambda'_i,\boldsymbol\lambda_i\in L_i,$ for $i=1,\dots, m$.\\
Similarly suppose that a linear torus $G$ has a torus coset $zG\subset X$ which is maximal among torus cosets contained in $X$. Then $G=H_{\tilde\Lambda}$, where $\Lambda$ is a subgroup generated by vectors of type $\boldsymbol\lambda'_i-\boldsymbol\lambda_i$ with $\boldsymbol\lambda'_i,\boldsymbol\lambda_i\in L_i,$ for $i=1,\dots, m$.\end{lemma}
\begin{proof}Suppose that $\mathbf zG\subset X$ is maximal. Let us define $L_{i,\chi} = \{\boldsymbol\lambda\in L_i: \chi_{\boldsymbol\lambda}=\chi\}$. By Lemma \ref{lem:artin} we have $$\sum_{\boldsymbol\lambda\in L_{i,\chi}}a_{i,\boldsymbol\lambda}\mathbf z^{\boldsymbol{\lambda}}=0$$for every $i$ and $\chi$ such that $L_{i,\chi}\ne \emptyset$.\\
Let $\Lambda=\sum_{\chi}\sum_{i=1}^m\sum_{\boldsymbol{\lambda,\lambda}'\in L_{i,\chi}} \mathbb Z(\boldsymbol {\lambda}-\boldsymbol{\lambda}')$. By definition of $L_{i,\chi}$ we find $G<H_{\Lambda}$. On the other hand note that for all $\mathbf h\in H_\Lambda$ and all $\chi$ and $i$ the term $\mathbf h^{\boldsymbol\lambda}$ is always the same, say $h_{i,\chi}$, for all $\boldsymbol\lambda\in L_{i,\chi}$ by definition of $\Lambda$. Therefore we can write $$P_i(\mathbf {zh})=\sum_{\chi:L_{i,\chi}\ne\emptyset}\left(\sum_{\boldsymbol\lambda\in L_{i,\chi}}a_{i,\boldsymbol\lambda}\mathbf z^{\boldsymbol{\lambda}}\right)h_{i,\chi}=0$$ and find $\mathbf z H_\Lambda\subset X$. This holds in both cases. Since $\mathbf zG$ is a maximal coset in $X$ we have $G=H_\Lambda.$ in the first case. In the second case we note that $G$ is an irreducible subset of $H_\Lambda$ which contains the identity. Thus the connected component of the identity $H^0$ of $H_\Lambda$ contains $G$. Thus we have $\mathbf zG\subset \mathbf z H^0\subset X$. Since $H^0$ is a linear torus and by the maximality of $\mathbf zG$ among torus cosets in $X$ we find $G=H^0$. Since $H_{\tilde\Lambda}$ is also a linear torus inside $H_\Lambda$ with finite index, we have $G=H_{\tilde\Lambda}$.\end{proof}

\begin{corollary}\label{cor:surj}Every algebraic subgroup $G$ of $\mathbb G_m^n$ is of type $H_\Lambda$ for some subgroup $\Lambda$ of $\mathbb Z^n$. \end{corollary}\begin{proof}Apply Lemma \ref{lem:3214} to $G$ and its maximal coset $G$.\end{proof}

\begin{corollary}\label{cor:fin}Let $X\subset \mathbb G_m^n$ be an algebraic subset. Then there are only finitely many linear tori $G<\mathbb G_m^n$ with the property that there exists $\mathbf x\in \mathbb G_m^n$ such that $\mathbf x G$ is a maximal torus coset in $X$.\end{corollary}
\begin{proof}Assume that $\mathbf x G$ is a maximal torus coset in $X$. By Lemma \ref{lem:3214} there exists $\Lambda<\mathbb Z^n$ generated by vectors of type $\boldsymbol\lambda'-\boldsymbol\lambda$ with $\boldsymbol\lambda',\boldsymbol\lambda$ in a finite set $L$ such that $G=H_{\tilde\Lambda}$. Since the generators come from a finite set, there are also only finitely many $\Lambda$ and hence $G$ is from a finite set.\end{proof}

\begin{definition}Let $\Lambda<\mathbb Z^n$ be a subgroup and $p$ a prime or $0$. The $p$-saturation $\Lambda_p$ is defined by $$\Lambda_p=\{\mathbf v\in \mathbb Z^n: \text{ there exists }k\ge 0 \text{ in }\mathbb{Z}: p^k\mathbf v\in \Lambda\}\text{ if }p\text{ is a prime, and }\Lambda_0=\Lambda.$$\end{definition}
\begin{remark}The $p$-saturation of a subgroup $\Lambda$ is $p$-full.\end{remark}
\begin{lemma}\label{lem:Heq}Let $\Lambda<\mathbb Z^n$ and char$(K)=p\ge0.$ We have $H_\Lambda=H_{\Lambda_p}.$\end{lemma}
\begin{proof}If $p=0$ we have $\Lambda_0=\Lambda$ and the equality obviously holds. Thus we may assume that $p>0$. Since $\Lambda\subset\Lambda_p$ it suffices to show $H_\Lambda\subset H_{\Lambda_p}$. Let $\mathbf x\in H_\Lambda$ and $\mathbf v\in \Lambda_p$. There exists a natural number $k$ such that $p^k\mathbf v\in \Lambda$. Therefore $\mathbf x^{p^k\mathbf v}=1$. Since the characteristic of the field is $p>0$ we have $(x+y)^{p^k}=x^{p^k}+y^{p^k}$ and hence we find $0=\mathbf{x^v}^{p^k}-1^{p^k}=(\mathbf {x^v}-1)^{p^k}$ which implies that $\mathbf {x^v}=1$ and thus $\mathbf x\in H_{\Lambda_p}$.\end{proof}

\begin{theorem}\label{thm:bij}The map $\Lambda\mapsto H_{\Lambda}$ is a bijection between $p$-full subgroups of $\mathbb Z^n$ and the algebraic subgroups of $\mathbb G_m^n$. The inverse is given by $G\mapsto \Lambda_G$. Moreover we have $\mathrm{rk}(\Lambda)+\dim(H_\Lambda)=n$ and $H_\Lambda$ is a linear torus if and only if $\Lambda$ is primitive.\end{theorem}
\begin{proof}Corollary \ref{cor:surj} and Lemma \ref{lem:Heq} show that the map is surjective. For injectivity assume that $\Lambda$ and $\Lambda'$ are two $p$-full subgroups of $\mathbb Z^n$ such that $H_\Lambda=H_{\Lambda'}$. By Lemma \ref{lem:HLambda} we have $\Lambda=\Lambda_{H_\Lambda}=\Lambda_{H_{\Lambda'}}=\Lambda'$. This proves bijectivity.
The other statements follow directly from Lemma \ref{lem:HLambda}.\end{proof}

\begin{definition}For a prime $p$ let $R_p$ be the localisation of $\mathbb Z$ at the powers of $p$.\end{definition}

\begin{lemma}\label{lem:bij}Let $n\in\mathbb{N}$ and $p$ be a prime. Then there is a bijection between the submodules $M$ of $R_p^n$ and the $p$-full subgroups $\Lambda$ of $\mathbb Z^n$. The maps are given by $M\mapsto M\cap\mathbb Z^n$ and $\Lambda\mapsto R\Lambda$.\end{lemma}
\begin{proof}We abbreviate $R=R_p$. First we show that both maps are well defined. Let $M$ be a submodule of $R^n$. Clearly $M\cap\mathbb Z^n$ is a subgroup of $\mathbb{Z}^n$. Assume that $p\mathbf v\in M\cap \mathbb{Z}^n$ for some $\mathbf v\in \mathbb{Z}^n$. There exists $\mathbf m\in M$ such that $p\mathbf{ v=m}$ and therefore $\mathbf v=\mathbf m/p\in M$ since $1/p\in R$ and hence $\mathbf v\in M\cap \mathbb{Z}^n$. Thus the first map is well defined. That $R\Lambda$ is a submodule of $R^n$ is straight forward.\\
Now we look at the compositions of the maps and show first that $R(M\cap\mathbb{Z}^n)=M$. The inclusion $\subset$ is immediate since $RM = M$. For the other direction let $\mathbf m\in M$. Then there exists $r\in\mathbb{Z}$ and $\mathbf v\in\mathbb{Z}^n$ such that $\mathbf m=p^r\mathbf v$. We have $\mathbf v= p^{-r}\mathbf m\in M$, so $\mathbf v\in M\cap\mathbb{Z}^n$ and hence $\mathbf m=p^r\mathbf v\in R(M\cap \mathbb Z^n)$.
Secondly we want to show that $(R\Lambda)\cap\mathbb Z^n=\Lambda$. Here the inclusion $\supset$ is direct. For the other direction let $\mathbf w=r\mathbf v\in (R\Lambda)\cap\mathbb Z^n$ where $r\in R$ and $\mathbf v\in\Lambda$. If $r\in\mathbb{Z}$, then $\mathbf w\in\Lambda$. Else we write $r=p^{-s}k$ for an integer $k$ and $s\in\mathbb{N}$. Then we find that $p^{s}\mathbf w = k\mathbf v\in\Lambda$. Since $\Lambda$ is $p$-full we find also in this case that $\mathbf w\in \Lambda$.
\end{proof}

\begin{corollary}If char$(K)=p>0$, the map $M\mapsto H_{M\cap \mathbb Z^n}$ is a bijection between the submodules of $R_p^n$ and the algebraic subgroups of $\mathbb{G}_m^n$.\end{corollary}\begin{proof}Combine Theorem \ref{thm:bij} and Lemma \ref{lem:bij}.\end{proof}

\begin{definition}A monoidal transform $\psi: \mathbb G_m^n\to \mathbb G_m^n$ is an automorphism given by $\mathbf x\mapsto\mathbf x^A$ for some $A \in \mathrm{GL}_n(\mathbb Z).$\end{definition}

\begin{lemma}\label{lem:redTorsCoset}Let $C\subset \mathbb G_m^n$ be a torsion coset of dimension $d$. Then there exists a monoidal transformation $\psi$ and $\boldsymbol\eta\in(\mathbb G_m^{n-d})_{\mathrm{tors}}$ such that $\psi(C)=\{\boldsymbol\eta\}\times \mathbb G_m^d.$\end{lemma}
\begin{proof}Let $G$ be the linear torus in $\mathbb G_m^n$ and $\boldsymbol\zeta\in(\mathbb G_m^n)_{\mathrm{tors}}$ such that $C=\boldsymbol\zeta G$. Let $\Lambda=\Lambda_G$. Let $B=(\mathbf b_1,\dots, \mathbf b_n)= (B_l,B_r)\in\mathbb \mathrm{GL}_n(\mathbb Z)$ such that $B_l\mathbb Z^{n-d} = \Lambda$. Let $\psi:\mathbf x\mapsto \mathbf x^B$ and $\boldsymbol{\eta = \zeta}^{B_l}$. Then using $B^{-1}$ one can show that $\psi(\boldsymbol\zeta G)=\{\boldsymbol\eta\}\times \mathbb G_m^d$ doing straight forward computations. \end{proof}

\section{The gcd of the determinats of the minors of a matrix}\label{sec:gcd}
The aim of this apppendix is to give two lemmas which imply that the product of the first $k$ entries of the diagonal matrix in the Smith Normal Form of $A$ is equal to the gcd of the determinats of the $k$-minors of $A$. This is used to bound the number of isolated torsion points.

\begin{definition}For $n,m\in \mathbb{N}$, a matrix $A=(a_{i,j})\in\mathrm{Mat}_{n\times m}(\mathbb{Z})$, $\emptyset\neq I\subset\{1,\dots, n\}$ and $\emptyset\neq J\subset\{1,\dots, m\}$ we define ${}_IA$ as the $\#I\times m$ matrix with entries $a_{i,j}$ where $i$ and $j$ range over $I$ and $\{1,…,n\}$, respectively, and in increasing
order. Similarly $A_J$ denotes the $n\times \#J$ matrix with entries $a_{i,j}$ where $i$ and $j$ range over $\{1,…,n\}$ and $J$, respectively, and in increasing order. Finally we see that $({}_IA)_J={}_I(A_J)$ and denote this matrix by ${}_IA_J$.\end{definition}
\begin{remark}For natural numbers $k,m,n,$ matrices $A\in\mathrm{Mat}_{k\times m}(\mathbb{Z}), B\in\mathrm{Mat}_{m\times n}(\mathbb{Z}), \emptyset\neq I\subset\{1,\dots, k\}$ and $\emptyset\neq J\subset\{1,\dots, n\}$ we have ${}_I(AB)={}_IAB$ and $(AB)_J=AB_J$.\end{remark}

\begin{definition}For $n,m\in\mathbb{N}$ and $A\in \mathrm{Mat}_{n\times m}(\mathbb{Z})$ let us define for $k=1,\dots,\min\{n,m\},$ $$d_k(A)=\gcd\{\det({}_IA_J): I\subset \{1,\dots n\}, J\subset \{1,\dots m\},\#I = k = \#J\}\in\mathbb{Z}_{\ge 0}$$ the gcd of the determinants of the $k\times k$-minors of $A$.\end{definition}

\begin{theorem}\label{thm:CauchyBinet} Let $m,n\in\mathbb{N}, A\in\mathrm{Mat}_{n\times m}(\mathbb{Z})$ and $B\in\mathrm{Mat}_{m\times n}(\mathbb{Z})$. Then we have $$\det(AB)=\sum_{\substack{I\subset \{1,\dots, m\}\\ \#I = n}} \det(A_I)\det({}_IB).$$\end{theorem}
\begin{proof}This is the Cauchy-Binet formula. A proof can be found in \cite{CB}, page 27-28.\end{proof}

\begin{lemma}\label{lem:diInv}Let $n,m\in\mathbb{N}, A\in \mathrm{Mat}_{n\times m}(\mathbb{Z})$, $P\in\mathrm{GL}_n(\mathbb{Z})$  and $Q\in\mathrm{GL}_m(\mathbb{Z})$. Then for all $k=1,\dots, \min\{n,m\}$ we have $d_k(PAQ)=d_k(A)$.\end{lemma}
\begin{proof}Since $P,Q$ are invertible it is enough to show that $d_k(A)\mid d_k(PAQ)$. So let $I\subset \{1,\dots n\}, J\subset \{1,\dots m\}$ such that $\#I = k = \#J$. By Theorem \ref{thm:CauchyBinet} we have $$\det({}_IPAQ_J)
=\sum_{\substack{S\subset \{1,\dots, n\}\\ \#S = k}} \det({}_IP_S)\det({}_S AQ_J)
=\sum_{\substack{S\subset \{1,\dots, n\}\\ \#S = k}} \det({}_IP_S)\sum_{\substack{T\subset \{1,\dots, m\}\\ \#T = k}} \det({}_SA_T)\det({}_TQ _J).$$
Therefore $d_k(A)\mid \det({}_IPAQ_J)$ and hence $d_k(A)$ divides their greatest common divisor $d_k(PAQ)$.
\end{proof}

\begin{definition}For $n,m\in\mathbb{N}$ and $\boldsymbol\alpha=(\alpha_1,\dots,\alpha_{\min\{n,m\}})\in\mathbb{Z}^{\min\{n,m\}}$ we denote by $\mathrm{diag}_{n\times m}(\boldsymbol\alpha)$ the $n\times m$ matrix $D=(d_{i,j})$ where $d_{i,j}=0$ if $i\neq j$ and $d_{i,i}=\alpha_i$ for all $i=1,\dots,\min\{n,m\}$.\end{definition}
\begin{lemma}\label{lem:diDiag}Let $n,m\in\mathbb{N}, l=\min\{n,m\}, \alpha_1,\dots,\alpha_l\in\mathbb{Z}_{\ge0}$ such that $\alpha_{i+1}\mathbb Z\subset \alpha_i\mathbb Z$ for all $i=1,\dots, l-1$. Then $d_k(\mathrm{diag}_{n\times m}(\alpha_1,\dots,\alpha_l))=\alpha_1\cdots\alpha_k$ for all $k=1,\dots, l$.\end{lemma}
\begin{proof}Let $k\in\{1,\dots,l\}, D=\mathrm{diag}_{n\times m}(\alpha_1,\dots,\alpha_l)$ and $I\subset \{1,\dots, n\},J\subset \{1,\dots, m\}$ such that $\#I=k=\#J$. If $I\neq J$ there is a row of zeros in ${}_I D_J$ since the entries are $d_{i,j}$ where $i\in I$ and $j\in J$ and hence the determinant vanishes. If $I=J$ the determinant of ${}_ID_J$ is given by $\prod_{i\in I}\alpha_i$. Therefore we find that $d_k(D)=\gcd(\{\prod_{i\in I}\alpha_i: I\subset\{1,\dots, l\}, \#I=k\})=\alpha_1\dots\alpha_k$, where we use the divisibility properties of the $\alpha_i$.\end{proof}

\bibliographystyle{alpha}
\bibliography{fink}

\end{document}